\documentclass[11pt]{amsart}

\usepackage[
    paperheight=11in,
    paperwidth=8.5in,
    left=1in,
    top=1in,
    right=1in,
    bottom=1in
]{geometry}

\usepackage[
    linktocpage=true,
    linktoc=all,
    bookmarksdepth=2,
    colorlinks=true,
    linkcolor=black,
    citecolor=black,
    filecolor=black,
    urlcolor=black,
    pagebackref=false,
    pdfstartpage={1},
    pdfstartview={FitH},
    pdftitle={Normal Bundle Splitting Strata of Rational Curves in Toric Varieties},
    pdfauthor={Hikari Iwasaki},
    pdfsubject={},
    pdfcreator={},
    pdfproducer={},
    pdfkeywords={}
]{hyperref}

\makeatletter
\renewcommand{\l@subsection}{\@tocline{2}{0pt}{3pc}{5pc}{}}
\makeatother

\usepackage{graphicx}
\usepackage{amsthm,amsfonts,amssymb,amscd}
\usepackage{cancel,comment,float}
\usepackage{enumerate}
\usepackage{euscript}
\usepackage{cite}
\usepackage{tikz}
\usepackage{multicol}
\usepackage{subcaption}

\newif\ifconsolidationcolor
\consolidationcolortrue

\newif\ifprelimtrimcolor
\prelimtrimcolortrue

\newif\ifprelimdefinitioncolor
\prelimdefinitioncolortrue

\ExplSyntaxOn
\clist_new:N \g_paperbib_cited_clist
\seq_new:N \l_paperbib_entries_seq
\tl_new:N \l_paperbib_preamble_tl
\tl_new:N \l_paperbib_selected_tl
\cs_set_protected:Npn  #1
  { \clist_gput_right:Nn \g_paperbib_cited_clist {#1} }

\cs_new_protected:Npn \paperbib_select:nn #1#2
  {
    \clist_if_in:NnTF \g_paperbib_cited_clist {*}
      { \tl_put_right:Nn \l_paperbib_selected_tl {#2} }
      {
        \clist_if_in:NnT \g_paperbib_cited_clist {#1}
          { \tl_put_right:Nn \l_paperbib_selected_tl {#2} }
      }
  }
\cs_new_protected:Npn \paperbib_item:w #1#2 \q_stop
  { \paperbib_select:nn {#1} { \bibitem{#1} #2 } }
\cs_new_protected:Npn \paperbib_labeled_item:w [#1]#2#3 \q_stop
  { \paperbib_select:nn {#2} { \bibitem[#1]{#2} #3 } }
\cs_new_protected:Npn \paperbib_collect:n #1
  {
    \tl_if_head_eq_charcode:nNTF {#1} [
      { \paperbib_labeled_item:w #1 \q_stop }
      { \paperbib_item:w #1 \q_stop }
  }

\AddToHook{begindocument/end}
  {
    \cs_new_eq:NN \paperbib_begin:n \thebibliography
    \cs_new_eq:NN \paperbib_end: \endthebibliography
    \RenewDocumentEnvironment{thebibliography}{m +b}
      {
        \seq_set_split:Nnn \l_paperbib_entries_seq { \bibitem } {#2}
        \seq_pop_left:NN \l_paperbib_entries_seq \l_paperbib_preamble_tl
        \tl_clear:N \l_paperbib_selected_tl
        \seq_map_inline:Nn \l_paperbib_entries_seq
          { \paperbib_collect:n {##1} }
        \tl_if_empty:NF \l_paperbib_selected_tl
          {
            \paperbib_begin:n {#1}
            \tl_use:N \l_paperbib_preamble_tl
            \tl_use:N \l_paperbib_selected_tl
            \paperbib_end:
          }
      }
      {}
  }
\ExplSyntaxOff

\usetikzlibrary{matrix,backgrounds}

\newtheorem{theorem}{Theorem}[section]
\newtheorem{lemma}[theorem]{Lemma}
\newtheorem{lemmadefinition}[theorem]{Lemma--Definition}
\newtheorem{corollary}[theorem]{Corollary}

\newtheorem{proposition}[theorem]{Proposition}

\theoremstyle{remark}
\newtheorem{remark}[theorem]{Remark}

\theoremstyle{definition}

\newtheorem*{question}{Question}
\newtheorem{eg_no_qed}[theorem]{Example}
\newtheorem{definition}[theorem]{Definition}

\AfterEndEnvironment{ex}{\noindent\ignorespaces}

\numberwithin{equation}{section}

\theoremstyle{plain}
\newtheorem{innercustomgeneric}{\customgenericname}
\providecommand{\customgenericname}{}

\makeatletter
\newcommand{\newcustomtheorem}[2]{%
  \newenvironment{#1}[1]
  {%
   \renewcommand\customgenericname{#2}%
   \def\theinnercustomgeneric{##1}%
   \@ifnextchar[{\innercustomgeneric}{\innercustomgeneric}%
  }
  {\endinnercustomgeneric}
}
\makeatother

\newcustomtheorem{customthm}{Theorem}
\newcustomtheorem{customconj}{Conjecture}
\newcustomtheorem{customprop}{Proposition}
\newcustomtheorem{customlem}{Lemma}

\newcommand{\Z}{\mathbb{Z}}
\newcommand{\PP}{\mathbb{P}}

\renewcommand{\ker}{\operatorname{ker}}

\newcommand{\ev}{\operatorname{ev}}

\newcommand{\cA}{\mathcal{A}}

\newcommand{\cE}{\mathcal{E}}
\newcommand{\cF}{\mathcal{F}}

\newcommand{\cJ}{\mathcal{J}}

\newcommand{\cN}{\mathcal{N}}
\newcommand{\cO}{\mathcal{O}}

\newcommand{\cV}{\mathcal{V}}

\tikzset{partition/.style={fill,circle,inner sep=1pt}}

\usetikzlibrary{positioning}
\usetikzlibrary{decorations.pathreplacing}
\usetikzlibrary{matrix,arrows}
\usetikzlibrary{calc}
\usetikzlibrary{shapes,arrows}
\usetikzlibrary{snakes}

\tikzset{
    partition/.style={fill,circle,inner sep=1pt},
    part/.style={baseline=0,scale=0.5,bend left=45},
    partlabel/.style={below}
}

\tikzstyle{pnt}=[draw,ellipse,fill,inner sep=1pt]
\tikzstyle{opnt}=[ ]
\tikzstyle{pntt}=[draw,ellipse,fill,inner sep=0.5pt]
\tikzstyle{point}=[draw,ellipse,fill,inner sep=2pt]

\title{Normal Bundle Splitting Strata of Rational Curves in Toric Varieties}
\author{Hikari Iwasaki}
\date{September 30, 2026}

\begin{document}
\begin{abstract}
We study stratification by normal-bundle splitting type in families of unramified rational curves of fixed class in a smooth projective toric variety over an algebraically closed field {of arbitrary characteristic}.
Using the Cox construction, we construct an explicit two-term resolution of the relative normal bundle by direct sums of line bundles on the universal curve over the fixed-source parameter space. Then we construct morphisms of vector bundles on the parameter space, which we call relative cohomology morphisms, from which we can investigate generic splitting type and jumping loci of the family. In particular, we obtain sufficient numerical criteria for unbalancedness of the normal bundle.
We apply this method to blowups of projective space along linear subspaces, and show that the numerical criteria for unbalancedness are equivalent to those of Cela--Lian \cite{CelaLian2026}. We also obtain explicit formulas for the expected Chow classes of the jump loci.

\end{abstract}
\maketitle
\tableofcontents
\section{Introduction}
\label{sec:introduction}
\subsection{Motivation}

Throughout this paper, we work over an algebraically closed field \(k\). A rational curve \(f : \PP^1 \to X\) in a smooth projective variety \(X\) of dimension \(n\) is \emph{unramified} if the differential \(df:T_{\PP^1}\to f^*T_X\) is injective on every fiber. Its \emph{normal bundle} is then defined by
\[N_f = \operatorname{coker}( df : T_{\PP^1} \to f^*T_X).\]  By the Birkhoff-Grothendieck Theorem~\cite{Grothendieck1957}, the normal bundle \(N_f\) splits into a sum of line bundles of the form \(\bigoplus_{i=1}^{n-1}\cO_{\PP^1}(a_i)\) for some integers \(a_i\).
The collection \(\mathbf{a} = (a_1,\cdots, a_{n-1})\) of integers, listed in nondecreasing order \(a_1\le\cdots\le a_{n-1}\), is called the \emph{splitting type} of \(N_f\); in this paper, we also call it the \emph{splitting type} of \(f\). When the integers \(a_i\) satisfy \(|a_i - a_j | \leq 1\) for all \(i,j\), the vector bundle is said to be \emph{balanced}.
More generally, when the integers \(a_i\) satisfy \(|a_i - a_j | \leq m\) for all \(i,j\), the vector bundle is said to be \emph{\(m\)-balanced}.
Given a family of unramified rational curves in \(X\) with fixed curve class, we are interested in the following types of questions. 
\par\noindent\begin{minipage}{\linewidth}
\begin{question}
  Let \(U\) be the family of unramified rational curves \(f: \PP^1 \to X\) of fixed curve class.
  \begin{enumerate}
    \item\label{que:general-splitting-type} What is the splitting type of a general member of the family \(U\)? Such a splitting type is called a \emph{general splitting type} of the family.
    \item\label{que:jumping-locus} What is the locus of maps \(f\in U\) whose splitting type deviates from the general one? Such a locus is called the \emph{first jumping locus} of the family.
  \end{enumerate}
\end{question}
\end{minipage}\par\vspace{6pt}

{These questions have been studied extensively for projective space, and also in Grassmannians and in Fano complete intersections. Some of the arguments are crucially dependent on the characteristic of the underlying field \(k\).}

\subsubsection*{Projective spaces}

For $d\ge n$, the normal bundle of a general nondegenerate rational curve of degree $d$ in $\PP^n$ is balanced in characteristic different from $2$ by a theorem of Larson--Vogt \cite{LV23}. In characteristic zero, Sacchiero determined the possible splitting types \cite{Sac80}. The nonempty splitting strata of smooth rational curves in $\PP^3$ over $\mathbb{C}$ are irreducible and have the expected dimension, by a result of Eisenbud--Van de Ven \cite{EisenbudVanDeVenStrata}; in higher dimension, irreducibility and expected dimension can fail as shown by Alzati--Re and Coskun--Riedl \cite{AR17,CoskunRiedl}. Ran studied enumerative counts of normally jumping curves in incidence pencils \cite{Ran}.

\subsubsection*{Grassmannians and Fano complete intersections}

For a general rational curve in a Grassmannian, Coskun--Larson--Vogt prove that the normal bundle is $2$-balanced \cite{CLV24}.
For Fano complete intersections, Coskun--Riedl proved the following existence results~\cite{CRci}: for a general Fano complete intersection of multidegree \((d_1,\ldots,d_a)\) in \(\PP^n\), with each \(d_i\ge2\), there are rational curves with balanced normal bundle in every degree \(1\le e\le n\) if some \(d_i\ge3\). If all \(d_i=2\) and \(n\ge2a+1\), their theorem gives every degree \(1\le e\le n-1\).

\subsubsection*{Special behavior in characteristic 2}

In characteristic $2$, an unramified map $f:\PP^1\to\PP^n$ of degree $d$ has the property that every summand of $N_f^\vee\otimes\mathcal O_{\PP^1}(d)$ has even degree, by Frobenius descent, as shown by Coskun--Larson--Vogt \cite[Section~3]{CoskunLarsonVogt2022}. Together with the generic $2$-balancedness theorem, this determines the generic normal splitting in projective space \cite{CLV24}. For example, a general rational quartic in \(\PP^3\) has normal bundle \(\cO(6)\oplus\cO(8)\) in characteristic two, rather than the generic type \(\cO(7)^{\oplus2}\) in characteristic different from two.

\subsection{Methods}
In this paper, we give partial answers to Questions~\hyperref[que:general-splitting-type]{(\ref*{que:general-splitting-type})} and~\hyperref[que:jumping-locus]{(\ref*{que:jumping-locus})} in the case that  \(X\) is a smooth projective toric variety. We construct morphisms of vector bundles on the parameter space \(U\), from which one can determine the general splitting type and the jumping locus of the family.
More precisely, we use the following cohomological approach for determining the splitting types of an unramified rational curve \( f : \PP^1 \to X\). Observe that if \(n\ge2\) and the degree of the normal bundle is \(D\), then the normal bundle is balanced if and only if it is of the form
\[N_f \cong \cO_{\PP^1}(c)^{\oplus (n-1-b)} \oplus \cO_{\PP^1}(c+1)^{\oplus b}\]
where \(c = \left\lfloor \frac{D}{n-1} \right\rfloor\) and \(b\)  is the remainder in the Euclidean division \(D = (n-1)c + b\) with \( 0 \leq b < n-1\). It follows that the cohomology dimensions of twists of \(N_f\) determine its splitting type. For example, the normal bundle is balanced if and only if  \(h^0(N_f(-c-2)) = h^1(N_f(-c-1)) = 0\).

Now, suppose that the normal bundle \(N_f\) fits into the short exact sequence of vector bundles \(0 \to E_f \to F_f \to N_f \to 0\) on \(\PP^1\). The induced long exact sequence of cohomology of the twist by a line bundle \(L = \cO(-c-1)\) is
\begin{equation}
  \label{eq:LES_fiber}
  \begin{aligned}
  0 &\to H^0(E_f(-c-1)) \xrightarrow{\theta_L^0} H^0(F_f(-c-1)) \to H^0(N_f(-c-1)) \\
  &\to H^1(E_f(-c-1)) \xrightarrow{\theta_L^1} H^1(F_f(-c-1)) \to H^1(N_f(-c-1)) \to 0.
  \end{aligned}
\end{equation}
By computing the cohomology dimensions of \(E_f(-c-1)\) and \(F_f(-c-1)\) and the rank of the induced morphism \(\theta_L^1\), we can determine the cohomology dimensions of \(N_f(-c-1)\).

 In this paper, we relativize this cohomological technique to the family \(U\) of unramified rational curves in \(X\). More precisely, let \(F:\PP^1\times U\to X\) be the universal curve map, whose restriction to \(\PP^1\times\{[f]\}\) is \(f\).  Then the relative normal bundle \(\cN\) of \(F\) is given as the cokernel of the relative tangent map, fitting into the relative tangent exact sequence
\[ 0 \to T_{\PP^1 \times U / U} \xrightarrow{dF} F^*T_X \to \cN \to 0. \]
Although this sequence can be regarded as a resolution of the relative normal bundle \(\cN\)  by vector bundles, the tangent bundles are generally not split. Our first main result, Theorem~\ref{thm:two-term-resolution-N}, gives a concrete two-term resolution of \(\cN\) by direct sums of line bundles.
\begin{theorem}[Informal Statement of Theorem~\ref{thm:two-term-resolution-N}]
There is a two-term resolution of the relative normal bundle \(\cN\) of the universal curve \(F : \PP^1 \times U \to X \),
\begin{align*}
  0 \to \cE \xrightarrow{\Psi} \cF \to \cN \to 0
\end{align*}
on \(\PP^1 \times U\), where \(\cE\) and \(\cF\) are direct sums of line bundles on \(\PP^1 \times U\).
\end{theorem}

Using the resolution, we construct morphisms \(\Theta_L^1\) of vector bundles on \(U\) for a line bundle \(L\) on \(\PP^1\), called \emph{relative cohomology morphisms}, whose fiber at \([f]\) is the morphism \(\theta_L^1\) in the fiberwise cohomology sequence for \(N_f\) given in \eqref{eq:LES_fiber}. Because \(\Psi\) is a morphism of direct sums of line bundles, \(\Theta_L^1\) is also a morphism of direct sums of line bundles.

\begin{lemmadefinition}[Informal Statement of Lemma--Definition~\ref{thm:universal-pushforward}]
Denote the first and second projections from \(\PP^1 \times U\) to the corresponding factors by \(\pi_1\) and \(\pi_2\), respectively. Let \(L\) be a line bundle on \(\PP^1\). 
The \emph{relative cohomology morphism} associated with \(L\) is given by 
 \[\Theta_L^1=R^1\pi_{2*}(\Psi\otimes\mathrm{id}_{\pi_1^*L}) :R^1\pi_{2*}(\cE\otimes\pi_1^*L) \to 
    R^1\pi_{2*}(\cF\otimes\pi_1^*L).  \] 
     \(\Theta_L^1\) is a morphism of direct sums of line bundles on \(U\), which fit into the following exact sequence
\begin{align*}
    \pi_{2*}(\cN\otimes \pi_1^*L) 
  &\to R^1\pi_{2*}(\cE\otimes \pi_1^*L) \xrightarrow{\Theta_L^1} R^1\pi_{2*}(\cF\otimes \pi_1^*L) \to R^1\pi_{2*}(\cN\otimes \pi_1^*L) \to  0.
\end{align*}
\end{lemmadefinition}

  Using the relative cohomology morphisms, we can answer Questions (1) and (2) above.
Section~\ref{subsec:toric-unbalanced-generalization} addresses Question~\hyperref[que:general-splitting-type]{(\ref*{que:general-splitting-type})} on general splitting type by deriving sufficient numerical conditions for the general normal bundle to be unbalanced. First, Lemma~\ref{lem:rank-deficiency-implies-unbalancedness} shows that failure of maximal rank of \(\Theta_L^1\) at a point of \(U\) forces the corresponding normal bundle to be unbalanced. Theorem~\ref{thm:toric-unbalanced-criterion} then gives an explicit sufficient condition forcing this failure of maximal rank at every point of \(U\). 

\begin{theorem}
With the notation of Section~\ref{subsec:toric-unbalanced-generalization}, if there exist \(c\ge2\) and a subspace \(W\subseteq V\) such that
\[
\delta_W(c)>\max\{\delta_{\mathrm{exp}}(c),0\},
\]
then the morphism \(\Theta_{L_c}^1\) fails to have maximal rank at every point of \(U\).
\end{theorem}
For varieties of Picard rank two satisfying the additional hypotheses of Proposition~\ref{prop:optimal-Wpi-projective-bundle}, testing this condition over subspaces of \(\operatorname{Pic}(X)^\vee\otimes_{\Z}k\) reduces to testing a single subspace \(W_{\pi}\) determined by a well-behaved toric morphism. 

Next, to address Question~\hyperref[que:jumping-locus]{(\ref*{que:jumping-locus})}, Section~\ref{sec:jumping_loci} introduces the cohomology-jump loci \(J_L^j\), where the rank of \(\Theta_L^1\) is at least \(j\) below the generic rank. These closed loci are defined by minors of explicit matrices in the Cox coefficients. Lemma~\ref{lem:jumping-generic-rank} recovers each prescribed splitting stratum from these rank conditions as the twist varies, and expresses the first jumping locus as the union of the first cohomology-jump loci. Because the comparison uses the actual generic rank, this description applies even when the general normal bundle is unbalanced. Proposition~\ref{prop:cohomology-jump-class} then applies Porteous's formula to compute the expected codimensions and Chow classes of the individual cohomology-jump loci from the explicit source and target bundles; these classes equal the fundamental classes when the loci have the expected codimension.



Finally, as an application, we apply these results to rational curves of fixed class in blowups of projective space along linear subspaces.  In this case, the parameter space \(U\) is given as a smooth quasiprojective variety, whose compactification is a projective bundle on a projective space, and we obtain an explicit construction of the morphisms \(\Theta_L^i\). In particular, for the family of quartic unramified rational curves in the blowup of \(\PP^3\) along a line meeting the exceptional divisor with multiplicity two, we show that the first jumping locus is nonempty and determine its divisor class in the Chow ring of the parameter space \(U\).

This paper is structured as follows. Section 2 introduces the preliminary construction for the parameter space \(U\) of unramified rational curves in a smooth projective toric variety  following the Cox coordinate construction in \cite{Cox1995, CoxFunctor1995} and the compactification technique introduced by Givental \cite[Section~5]{Givental1998} and Morrison--Plesser \cite[Section~3.7]{MorrisonPlesser1995}.  In Section 3, we construct morphisms on the parameter space, called the \emph{relative cohomology morphisms},  from which we can analyze answers to the two questions. We start with Theorem \ref{thm:two-term-resolution-N} in Section 3.1, in which we construct the two-term resolution of the relative normal bundle of the universal curve over the parameter space using Cox coordinates for toric varieties. In Section 3.2, we use such a resolution to construct the relative cohomology morphisms, by twisting the resolution and pushing forward to the parameter space. Sections 4 and 5 answer each of the two questions using the relative cohomology morphisms.  In Section 4, we deduce sufficient numerical criteria for unbalancedness of the general normal bundle. Section 5 provides a method to answer the complementary question on jumping loci. Section 6 applies results from Sections 4 and 5 to blowups of projective space along linear subspaces. In particular, in Proposition~\ref{prop:blowup-cela-lian-equivalence}, we show that the sufficient numerical condition derived in Theorem~\ref{thm:toric-unbalanced-criterion} is sharp under the hypotheses stated in Section~\ref{subsec:blowup-unbalanced-criterion}: it is equivalent to the unbalancedness condition given by Cela and Lian \cite{CelaLian2026}. We also derive the Chern class of the jumping loci when they have expected codimension. {We present an example of quartic curves with intersection multiplicity two with the exceptional divisor in the blowup of \(\mathbb{P}^3\) along a line; we show that the first jumping locus has expected codimension, and compute the class of the resulting jumping divisor in every characteristic.}




\subsection*{AI Disclosure}
The author used GPT-6 Sol/Astra as an assisting tool for learning background material, searching and reviewing the literature, and revising the manuscript. The model suggested the point \(q_0\) in Proposition~\ref{prop:small-smooth-witness} as a witness to nonemptiness of the first jumping locus; the author independently verified the example and supplied its proof. The general mathematical arguments and the writing of the paper are produced by the author, and the author is fully responsible for the correctness of the paper.

\subsection*{Acknowledgments}
I am grateful to Izzet Coskun and my advisor Ravi Vakil for helpful discussions and feedback on the draft. 

\section{Preliminaries}
\label{sec:parameter-space-preliminaries}

In this section, we briefly review the construction of Cox coordinates for a smooth projective toric variety \(X\), following \cite[Sections~1--2]{Cox1995}. We then use Cox's description of morphisms~\cite{CoxFunctor1995} to construct the parameter space \(U\) of unramified morphisms \(\PP^1\to X\) with fixed integral curve class \(\beta\) and fixed source. The ambient quasimap compactification is described in \cite[Section~7.2, especially Example~7.2.3]{CiocanFontanineKim2010}. This construction provides the setting for our results on the relative normal bundle and its jumping loci.


Let \(X = X_\Sigma\) be a smooth projective toric variety associated with a fan \(\Sigma\) in \(\mathbb{R}^n\). Write \(\Sigma(j)\) for the set of \(j\)-dimensional cones and \(\sigma(1)\) for the rays of a cone \(\sigma\). In particular, \(\Sigma(1)\) is the set of rays, and \(\Sigma(n)\) is the set of maximal cones.
Recall that a ray \(\rho \in \Sigma(1)\) of the fan \(\Sigma\) corresponds to a torus-invariant divisor \(D_\rho\) in \(X\). 
The \emph{Cox ring} is the \(\operatorname{Pic}(X)\)-graded \(k\)-algebra
\[
  R_\Sigma=k[x_\rho\mid\rho\in\Sigma(1)],
  \quad \text{ where } \deg(x_\rho)=[D_\rho]\in\operatorname{Pic}(X).
\]
Its variables \(x_\rho\) are called \emph{Cox coordinates}; we also denote by \(x_\rho\) the corresponding canonical section of \(\cO_X(D_\rho)\).
The \emph{Cox torus} is defined by \(G = \operatorname{Hom}(\operatorname{Pic}(X), \mathbb{G}_m)\).
Define the open subscheme \(U_\Sigma\subset\operatorname{Spec}(R_\Sigma)\) by
\[
  U_\Sigma
  =\bigcup_{\sigma\in\Sigma(n)}
    \left\{ p\in\operatorname{Spec}(R_\Sigma)
  \mid
      x_\rho(p)\ne0
      \text{ for every }\rho\notin\sigma(1)
    \right\},
\]
where the Cox torus acts on \(U_\Sigma\) by 
\[g\cdot x_\rho= g([D_\rho])x_\rho \quad  \text{ for all } g \in G.\]
The geometric quotient of \(U_\Sigma\) by \(G\) then retrieves the toric variety: \(X_\Sigma\cong U_\Sigma/G\).

Next, we describe the parameter space \(U\) of unramified morphisms \(f\colon\PP^1\to X\) with fixed integral curve class \(\beta\) and fixed parametrization of the source, following Cox's homogeneous-polynomial description of morphisms.
For each ray \(\rho \in \Sigma(1)\), the pullback \(f^*x_\rho\) is a homogeneous polynomial \(u_\rho\) of degree \(d_\rho := D_\rho \cdot \beta\) in the homogeneous coordinates \(s,t\) of the source \(\PP^1\). 
The collection of \emph{Cox tuples}  \((u_\rho)_\rho\) of such polynomials forms a vector space \(\cA=\bigoplus_{\rho\in\Sigma(1)} V_\rho\), where \(V_\rho = H^0\!\left(\PP^1,\cO_{\PP^1}(d_\rho)\right).\)
Define the open Cox-stable subscheme \(\cA^\circ\) by
\[
  \cA^\circ
  =\bigcup_{\sigma\in\Sigma(n)}
    \left\{
    (u_\rho)_\rho\in\cA
      \,\middle|\,
      u_\rho\not\equiv0
      \text{ for every }\rho\notin\sigma(1)
    \right\},
\]
where the Cox torus acts on \(\cA\) blockwise by \[g\cdot(u_\rho)_\rho=\bigl(g([D_\rho])u_\rho\bigr)_\rho \quad \text{ for all }g \in G.\] The geometric quotient \(Q:=\cA^\circ/G\) is a projective compactification of the fixed-source maps of class \(\beta\), called the space of \emph{stable toric quasimaps} with fixed parametrization on the source.

The locus \(Q^{\mathrm{map}}\) of honest morphisms is an open subscheme of \(Q\), defined as follows. First, we let \(\cA^{\mathrm{map}}\subset \cA^\circ\) be the Cox-stable open subscheme of \(\cA^\circ\) consisting of Cox tuples which induce base-point-free morphisms \(f : \PP^1 \to X\), given by
\[\cA^{\mathrm{map}}= \left\{ (u_\rho)_\rho \in \cA^\circ \mid \text{for every }p \in \PP^1(k), \text{ there is }\sigma \in \Sigma(n) \text{ such that }u_\rho(p) \neq 0 \text{ for all }\rho \notin \sigma(1)\right\}.\]
The locus \(\cA^{\mathrm{map}}\) is \(G\)-invariant and open in \(\cA^\circ\). Indeed, the pairs \((p,(u_\rho)_\rho)\) for which \(p\) is a base point form a closed subset of \(\PP^1\times\cA^\circ\), defined by the vanishing of \(\prod_{\rho\notin\sigma(1)}u_\rho(p)\) for every \(\sigma\in\Sigma(n)\). Its image under the proper projection to \(\cA^\circ\) is closed, and its complement is \(\cA^{\mathrm{map}}\).
It follows that the quotient \(Q^{\mathrm{map}}=\cA^{\mathrm{map}}/G\) is an open subscheme of \(Q\), and is the parameter space of honest morphisms with curve class \(\beta\).
By Cox's functorial description of morphisms, \(Q^{\mathrm{map}}\) represents the Hom functor of morphisms \(\PP^1\to X\) of class \(\beta\) with fixed source. The \emph{universal curve} map
\[
  F^{\mathrm{map}}:\PP^1\times Q^{\mathrm{map}}\longrightarrow X
\]
is the universal morphism corresponding to the identity of \(Q^{\mathrm{map}}\) under this representability.

Finally, the space of unramified morphisms is defined by imposing the unramified condition on the parameter space \(Q^{\mathrm{map}}\).
\begin{definition}
\label{def:unramified-parameter-space}
The parameter space \(U\) of unramified morphisms \(\PP^1\to X\) of class \(\beta\) with fixed source is
\[
  U
  =\left\{ 
  q\in Q^{\mathrm{map}}
  \,\middle|\, 
  (df_q)_p\colon T_{\PP^1,p}\longrightarrow T_{X,f_q(p)} 
  \text{ is injective for every }p\in\PP^1_{k}
  \right\}, 
\]
where \(f_q : \PP^1 \to X\) is the restriction of \(F^{\mathrm{map}}\) to the fiber over \(q \in Q^{\mathrm{map}}\).
\end{definition}
The locus \(U\) is open in \(Q^{\mathrm{map}}\): the locus where the relative differential of \(F^{\mathrm{map}}\) fails to be injective on fibers is closed in \(\PP^1\times Q^{\mathrm{map}}\), and its image under the proper projection to \(Q^{\mathrm{map}}\) is the complement of \(U\). 
In Remark~\ref{rmk:unramified-locus-rank-condition}, we give an alternative construction of \(U\) using the relative normal sheaf of \(F^{\mathrm{map}}\).  Throughout the remainder of the paper, we assume \(U\) is nonempty. In particular, \(U\) is smooth and integral, being a nonempty open subscheme of the smooth integral Cox quotient \(Q=\cA^\circ/G\). We denote by \(F : \PP^1 \times U \to X\) the restriction of the universal curve map \(F^{\mathrm{map}}\) to \(\PP^1 \times U\).


\begin{remark}
All constructions in this section are valid in arbitrary characteristic of the underlying field. In particular, for morphisms \(\PP^1\to X\), unramifiedness is equivalent to fiberwise injectivity of the differential in every characteristic. The proper-projection argument above therefore establishes the openness of \(U\subset Q^{\mathrm{map}}\) without any restriction on the characteristic~\cite[Definition~5.4.1 and Theorem~5.5.3(i)]{EGAII}.
\end{remark}

\section{Relative cohomology morphism}
\label{sec:two-term-resolution-and-pushforward}

  In this section, we reduce the problem of identifying normal bundle splitting types and their jumping loci in \(U\) to explicit rank computations in Cox coordinates of the smooth projective toric variety. The first construction, Theorem~\ref{thm:two-term-resolution-N}, is a two-term resolution of the relative normal bundle of the universal curve by direct sums of line bundles, which holds valid in arbitrary characteristic. Twisting this resolution and pushing forward to \(U\) then yields the \emph{relative cohomology morphisms} \(\Theta_L^1\), defined in Lemma--Definition~\ref{thm:universal-pushforward}, which can be  represented in monomial bases by matrices whose entries are linear in the Cox coefficients. Their ranks determine cohomology dimensions of twists of the general normal bundle, and their minors give equations for the loci where this dimension increases. By varying the twists, these cohomology dimensions recover the normal splitting type. Using the relative cohomology morphisms, we approach Questions~\hyperref[que:general-splitting-type]{(\ref*{que:general-splitting-type})} and~\hyperref[que:jumping-locus]{(\ref*{que:jumping-locus})} in Section~\ref{subsec:toric-unbalanced-generalization} and Section~\ref{sec:jumping_loci}.

\subsection{Two-term resolution of the relative normal bundle}
\label{sec:normal-bundle-universal-curve}
We first resolve the relative normal bundle by direct sums of line bundles. For convenience, let \(\pi_i\) denote the projection from \(\PP^1 \times U\) to the \(i\)-th factor and write \(\cO(a,\zeta) := \pi_1^*\cO_{\PP^1}(a) \otimes \pi_2^*\cO_{U}(\zeta) \) for  \(\zeta\in\operatorname{Pic}(U)\). 
The \emph{relative normal bundle} of \(F\) is
\[
  \cN:=\operatorname{coker}\bigl(dF:
  T_{\PP^1\times U/U}\longrightarrow F^*T_X\bigr).
\]
For \(q\in U\), write \(N_q:=\cN|_{\PP^1\times\{q\}}\cong N_{f_q}\) for the normal bundle of the corresponding curve. Then \(\cN\) has the following two-term resolution by direct sums of line bundles.

\begin{theorem}
  \label{thm:two-term-resolution-N} Let \(X\) be a smooth projective toric variety of dimension \(n\) over \(k\) associated with a fan \(\Sigma\), with Picard \(\Z\)-rank \(m\), and let \(U\) be the space parameterizing unramified rational curves of fixed integral curve class \(\beta\). Write \(F:\PP^1\times U\to X\) for the universal curve and set
\[
 V=\operatorname{Pic}(X)^\vee\otimes_{\mathbb Z}k,
 \qquad \bar\beta=\beta\otimes1\in V,
 \qquad \cF=\bigoplus_{\rho\in\Sigma(1)}\cO(d_\rho,\zeta_\rho),
\]
where \(\operatorname{Pic}(X)^\vee=\operatorname{Hom}_{\mathbb Z}(\operatorname{Pic}(X),\mathbb Z)\), \(d_\rho=D_\rho \cdot \beta \in \Z\), and \(\zeta_\rho\) is the unique divisor class satisfying \(F^*\cO_X(D_\rho)\cong\cO(d_\rho,\zeta_\rho)\). Define the vector bundle
\[
 \cE=\operatorname{coker}\!\left(
 \cO\xrightarrow{(\bar\beta;-s,-t)}
 (V\otimes_{k}\cO)\oplus\cO(1,0)^{\oplus2}\right),
\]
which decomposes as
\[
\cE\cong
\begin{cases}
 \cO^{\oplus(m-1)}\oplus\cO(1,0)^{\oplus2},&\bar\beta\ne0,\\
 \cO^{\oplus m}\oplus\cO(2,0),&\bar\beta=0.
\end{cases}
\]
 Then there is a short exact sequence on \(\PP^1\times U\),
\begin{equation}
\label{eq:two-term-resolution-N}
 0\longrightarrow\cE\xrightarrow{\Psi}\cF
 \longrightarrow\cN\longrightarrow0.
\end{equation}
\end{theorem}

\begin{proof}
By the definition of unramifiedness, on \(\PP^1 \times U\), we have the normal exact sequence
\[ 0 \to {T_{\PP^1 \times U/U}} \to F ^* T_X \to \cN \to 0  \quad \text{ on }\PP^1 \times U.
\] 
We use the toric Euler sequences for \(\PP^1\) and \(X\) following the construction in \cite[Theorem~8.1.6 and Exercise~8.1.7]{CoxLittleSchenck2011} to relate the two tangent bundles. The toric Euler sequence for \(X\) is given by
\[0 \to V\otimes_k\cO_{X} \xrightarrow{\epsilon} \bigoplus_{\rho \in \Sigma(1)} \cO_X(D_\rho) \to T_{X} \to 0 ,\]
{where the first map is \(\epsilon(\lambda)=(\lambda([D_\rho])x_\rho)_\rho\), where \(x_\rho\) is the Cox section on \(X\) associated with the ray \(\rho\). } Pulling back the toric Euler sequence via \(F\), we therefore obtain
\[
  0 \to V\otimes_k\cO_{\PP^1\times U}
  \xrightarrow{\epsilon_F} \cF\longrightarrow F^*T_X\longrightarrow0,
  \quad \text{ where }
  \cF:=\bigoplus_{\rho\in\Sigma(1)}
  \cO(d_\rho, \zeta_\rho).
\]
On the other hand, we have a similar toric Euler sequence for \(\PP^1\), which gives a short exact sequence on \(\PP^1 \times U\) under the pullback by \(\pi_1\):
\begin{equation}
\label{eq:pulled-back-p1-euler}
0 \to \cO_{\PP^1 \times U} \xrightarrow{(s,t)} \cO(1,0)^{\oplus2} \to T_{\PP^1 \times U/U}\to 0  \quad \text{ on }\PP^1 \times U,
\end{equation}
where \(s,t\) are the standard homogeneous coordinates of \(\PP^1\). Now, we show that there exists a commutative diagram of the two exact sequences of the following form:
\begin{equation}
  \label{diag:two-row-euler-sequences}
  \begin{array}{ccccccccc}
    0&\longrightarrow&\cO_{\PP^1\times U}
     &\xrightarrow{(s,t)}&\cO(1,0)^{\oplus2}
     &\longrightarrow&T_{\PP^1\times U/U}&\longrightarrow&0\\
    &&\big\downarrow\alpha&&\big\downarrow J&&\big\downarrow dF\\
    0&\longrightarrow&V\otimes_k\cO_{\PP^1\times U}
     &\xrightarrow{\epsilon_F}&\cF
     &\longrightarrow&F^*T_X&\longrightarrow&0.
  \end{array}
\end{equation}
First, we define the middle vertical map \(J\) by the Jacobian on the Cox coordinates as
\[
  J=(D_s,D_t),\quad \text{ where }
  D_s=(\partial_su_\rho)_\rho,\quad
  D_t=(\partial_tu_\rho)_\rho.
\]
Then the right square commutes by the construction of \(J\). 
To make the left square commute, we note that by Euler's identity for homogeneous polynomials, we have \[sD_s+tD_t
=\bigl(s\partial_su_\rho+t\partial_tu_\rho\bigr)_\rho
=(d_\rho u_\rho)_\rho = \epsilon_F(\beta).\] In particular, setting \(\alpha(1)=\bar\beta\) makes diagram~\eqref{diag:two-row-euler-sequences} commute, as desired.

Next, by the mapping-cone construction and its associated long exact cohomology sequence \cite[Chapter~IV, Exercise~3]{CartanEilenberg1956}, we obtain the following long exact sequence:
\[
  \begin{gathered}
    0\longrightarrow
    H^{-1}\!\left(\operatorname{Cone}(dF)\right)
    \longrightarrow T_{\PP^1\times U/U}
    \xrightarrow{\ dF\ }F^*T_X
    \longrightarrow H^0\!\left(\operatorname{Cone}(dF)\right)
    \longrightarrow0.
    \begin{aligned}
    \end{aligned}
  \end{gathered}
\]
On the other hand, by the defining relative normal exact sequence for \(F\), we obtain \(H^{-1}(\operatorname{Cone}(dF)) = 0\) and \(H^{0}(\operatorname{Cone}(dF)) = \cN\). The mapping cone therefore yields a four-term exact sequence on \(\PP^1 \times U\):
\[
    0\longrightarrow\cO_{\PP^1\times U}
    \xrightarrow{ \ (\bar\beta; -s , -t) \ }
    \left(V\otimes_k\cO_{\PP^1\times U}\right)
      \oplus\cO(1,0)^{\oplus2}
    \xrightarrow{\ \widetilde{\Psi} = \epsilon_F + J \ }
    \cF
    \longrightarrow \cN\longrightarrow0.
\]
The first cone map \((\bar\beta;-s,-t)\) is an injection of subbundles because \(s\) and \(t\) have no common zero. When \(\bar\beta\ne0\), a nonzero constant component in \(V\) can be cancelled, so that the cokernel \(\cE\) of the first cone map is \(\cO^{\oplus(m-1)}\oplus\cO(1,0)^{\oplus2}\). When \(\bar\beta=0\), the cokernel \(\cE\) of the first cone map is isomorphic to \(\cO_{\PP^1\times U}^{\oplus m} \oplus \left(\cO(1,0)^{\oplus 2}/\operatorname{im}(-s,-t)\right)\). The Euler sequence identifies the last factor with \(\cO(2,0)\), leaving \( \cE \cong \cO^{\oplus m}\oplus\cO(2,0)\). In particular, taking the quotient by the first cone map gives \eqref{eq:two-term-resolution-N}, where unramifiedness identifies the remaining cokernel with \(\cN\).
\end{proof}

\begin{remark}
\label{rmk:intrinsic-interpretation-E-as-extension}
More intrinsically, the cokernel \(\cE\) is the extension of \(\cO(2,0)\) by \(V \otimes_k \cO\) associated with \(\bar\beta : \cO \xrightarrow{} V\otimes_k\cO\). Indeed, pushing out the exact sequence~\eqref{eq:pulled-back-p1-euler} along \(\cO \xrightarrow{\bar\beta} V\otimes_k\cO\) gives an exact sequence 
\[0 \to V \otimes_k \cO \xrightarrow{j} \cE' \to \cO(2,0) \to 0,\] 
where the middle term is isomorphic to
\[
\bigl((V\otimes\mathcal O)\oplus\mathcal O(1,0)^{\oplus2}\bigr)
/\operatorname{im}(\bar\beta;-s,-t) = \cE.
\]
Under this perspective, it is immediate that the extension splits if and only if \(\overline{\beta} = 0\).

\end{remark}

\begin{remark}
\label{rmk:matrix-Psi}
Choosing bases in the decomposition of \(\cE\), the morphism \(\Psi\) has the following matrix.
If \(\bar\beta\ne0\), choose a basis \(\lambda_1,\ldots,\lambda_{m-1}\) of a complement to \(k\bar\beta\) in \(V\). Then
\[
[\Psi]=\begin{bmatrix}
\lambda_1([D_\rho])u_\rho&\cdots&
\lambda_{m-1}([D_\rho])u_\rho&\partial_su_\rho&\partial_tu_\rho
\end{bmatrix}_\rho.
\]
If \(\bar\beta=0\), choose a basis \(\lambda_1,\ldots,\lambda_m\) of \(V\). Using the identification \(\cO(1,0)^{\oplus2}/\cO\cong\cO(2,0)\) induced by \((a,b)\mapsto ta-sb\), the Jacobian descends to a morphism \(\cO(2,0)\to\cF\). Its components are the unique sections \(H_\rho\in H^0(\PP^1\times U,\cO(d_\rho-2,\zeta_\rho))\) satisfying
\[\partial_su_\rho=tH_\rho,\quad\partial_tu_\rho=-sH_\rho.
\]
The matrix of \(\Psi\) is then
\[
[\Psi]=\begin{bmatrix}
\lambda_1([D_\rho])u_\rho&\cdots&\lambda_m([D_\rho])u_\rho&H_\rho
\end{bmatrix}_\rho.\]
\end{remark}

\begin{remark}
\label{rem:relative-normal-complex} The construction extends naturally as we enlarge the parameter space from \(U\) to \( Q^{\mathrm{map}}\) and  \(Q.\)
On \(\PP^1\times U\), the universal map is unramified, so the cokernel of its relative differential is the relative normal bundle \(\cN\).
On the larger honest-map locus \(Q^{\mathrm{map}}\), the universal evaluation map \(F^{\mathrm{map}}\) still exists, and we define the relative normal \emph{sheaf}
\[
\cN^{\mathrm{map}}
:=
\operatorname{coker}\!\left(
T_{\PP^1\times Q^{\mathrm{map}}/Q^{\mathrm{map}}}
\xrightarrow{dF^{\mathrm{map}}}
(F^{\mathrm{map}})^*T_X
\right).
\]
This coherent sheaf need not be locally free.

To extend the normal bundle to \(Q\), we first note that the universal curve map is not defined away from \(Q^{\mathrm{map}}\); the normal sheaf defined via the differential is ill-defined. The bundles \(\cE,\cF\) and the morphism \(\Psi\), however, extend to \(\PP^1\times Q\). We therefore define the relative normal \emph{complex} by
\[
\cN_Q^\bullet
:=
\left[
\cE_Q
\xrightarrow{\Psi_Q}
\cF_Q
\right],
\]
with terms in degrees \(-1\) and \(0\). On  \(Q^{\mathrm{map}}\), the mapping-cone construction gives a canonical quasi-isomorphism
\[
\left.\cN_Q^\bullet\right|_{\PP^1\times Q^{\mathrm{map}}}
\simeq
\left[
T_{\PP^1\times Q^{\mathrm{map}}/Q^{\mathrm{map}}}
\xrightarrow{dF^{\mathrm{map}}}
(F^{\mathrm{map}})^*T_X
\right].
\]
On \(\PP^1\times U\), the complex is quasi-isomorphic to the relative normal bundle \(\cN\), placed in degree zero.
\end{remark}

\begin{remark}
\label{rmk:unramified-locus-rank-condition}
Having described the relative normal sheaf on \(\PP^1 \times Q^{\mathrm{map}}\) as the cokernel of \(\Psi_{Q^{\mathrm{map}}}\), we can characterize the locus \(U\) in \(Q^{\mathrm{map}}\) using the notation introduced in the proof of Theorem~\ref{thm:two-term-resolution-N} as follows. For \(q\in Q^{\mathrm{map}}\) and \(x\in\PP^1_{k}\), let \(\Psi_q(x)\) denote the fiber of \(\Psi_{Q^{\mathrm{map}}}\) at \((x,q)\).  Since \(q\) represents an honest morphism, the fiber of \(\epsilon_{F^{\mathrm{map}}}\) at \((x,q)\) has full rank \(m\). By construction, the Jacobian \(J_{(x,q)}\) maps \(\operatorname{im}\bigl((s,t)_x\bigr)\) into \(\operatorname{im}(\epsilon_{F^{\mathrm{map}},(x,q)})\) and the induced morphism on the cokernel is exactly \( (df_q)_x\colon T_{\PP^1,x}\longrightarrow T_{X,f_q(x)}. \) Now, note that the image of \((\bar\beta; -s,-t)\) is contained in the kernel of \(\widetilde{\Psi}\). This implies that when we truncate the four-term sequence to obtain \(\Psi\) from \(\widetilde{\Psi}\), the image of \(\Psi\) remains the same as that of \(\widetilde{\Psi}\). It follows that \(\operatorname{rank}\Psi_q(x)=m+\operatorname{rank}(df_q)_x.\)  Because \(T_{\PP^1,x}\) is one-dimensional, \(f_q\) is unramified at \(x\) if and only if \(\Psi_q(x)\) has full rank, rank \(m + 1\).  Consequently,
  \[
    U
    =\left\{
      q\in Q^{\mathrm{map}}
      \;\middle|\;
      \operatorname{rank}\Psi_q(x)\geq m+1
      \text{ for every }x\in\PP^1_{k}
    \right\}.
  \]
By the lower semicontinuity of the rank of a morphism of vector bundles, the locus where the rank of \(\Psi_q(x)\) is at most \( m\) is closed in \(\PP^1\times Q^{\mathrm{map}}\). Since the projection to \(Q^{\mathrm{map}}\) is proper, its image is closed, and hence \(U\) is open.
\end{remark}


\subsection{Relative cohomology morphism}
\label{sec:normal-bundle-pushforward-sequence}

Using the two-term resolution of the relative normal bundle of \(\PP^1 \times U\), we construct morphisms of vector bundles on \(U\), called \emph{relative cohomology morphisms}, whose determinantal loci correspond to jumping loci of the normal bundle, defined in the following  Lemma--Definition~\ref{thm:universal-pushforward}. In particular, the relative cohomology morphisms are shown to be morphisms of direct sums of line bundles on \(U\),  which can be used to compute the determinantal ideals of the induced cohomology maps.

For the remainder of the paper, \(c\) denotes an integer, and we set \(L_c = \cO_{\PP^1}(-c-1)\).

\begin{lemmadefinition}
\label{thm:universal-pushforward}
The \emph{\(i\)-th relative cohomology morphism} associated with a line bundle \(L\) on \(\PP^1\) is given by
 \[\Theta_{L}^i=R^i\pi_{2*}(\Psi\otimes\mathrm{id}_{\pi_1^*L}) :R^i\pi_{2*}(\cE\otimes\pi_1^*L) \to
    R^i\pi_{2*}(\cF\otimes\pi_1^*L).  \]
    When \(i= 1\), we also simply call \(\Theta_{L}^1\) the \emph{relative cohomology morphism} associated with \(L\).
The morphisms fit into the exact sequence
  \begin{equation}
    \label{eq:universal-pushforward-LES}
    \begin{aligned}
    0\longrightarrow{}
    (\pi_2)_*(\cE\otimes\pi_1^*L)
    &\xrightarrow{\Theta^0_{L}}
    (\pi_2)_*(\cF\otimes\pi_1^*L)
    \longrightarrow
    (\pi_2)_*({\cN}\otimes\pi_1^*L)\\
    \longrightarrow{}
    R^1(\pi_2)_*(\cE\otimes\pi_1^*L)
    &\xrightarrow{\ \Theta^1_{L}\ }
    R^1(\pi_2)_*(\cF\otimes\pi_1^*L)
    \longrightarrow
    R^1(\pi_2)_*({\cN}\otimes\pi_1^*L)
    \longrightarrow0.
    \end{aligned}
  \end{equation}
    When \(L = L_c\) for some integer \(c\), then
     \(\Theta_{L_c}^i\) is a morphism of direct sums of line bundles, with source and target  given by
\begin{equation}
\label{eq:pushforward-bases}
\begin{aligned}
R^i\pi_{2*}(\cE\otimes\pi_1^*L_c)
&\cong
\begin{cases}
\bigl(H^i(\cO(-c-1))^{\oplus(m-1)}\oplus
H^i(\cO(-c))^{\oplus2}\bigr)\otimes_{k}\cO_U,
&\text{if }\bar\beta\ne0,\\
\bigl(H^i(\cO(-c-1))^{\oplus m}\oplus
H^i(\cO(1-c))\bigr)\otimes_{k}\cO_U,
&\text{if }\bar\beta=0,
\end{cases}\\
R^i(\pi_2)_*(\cF\otimes\pi_1^*L_c)
&\cong\bigoplus_{\rho\in\Sigma(1)}
H^i(\cO(d_\rho-c-1))\otimes_{k}\cO_U(\zeta_\rho),
\end{aligned}
\end{equation}
where all cohomology groups are on \(\PP^1\).
\end{lemmadefinition}

\begin{proof}


After twisting by the pullback to \(\PP^1 \times U\) of the line bundle \(L\), the pushforward to the parameter space \(U\) induces a six-term exact sequence \eqref{eq:universal-pushforward-LES}. Observe that \(\cE\) is the pullback under \(\pi_1\) of a vector bundle on \(\PP^1\). For any vector bundle \(\cV\) on \(\PP^1\), flat base change from \(k\) gives
\[
  R^i(\pi_2)_*(\pi_1^*\cV)
  \cong H^i(\PP^1,\cV)\otimes_{k}\cO_{U},
  \qquad i=0,1.
\]
Together with the projection formula, it follows that the higher pushforwards of \(\cE \otimes \pi_1^*L_c\) and \(\cF \otimes \pi_1^*L_c\) are given as in \eqref{eq:pushforward-bases}.
In particular, the maps \(\Theta^i_{L}\) are morphisms of direct sums of line bundles for all \(i\).
\end{proof}





\section{Sufficient numerical criteria for unbalancedness}
\label{subsec:toric-unbalanced-generalization}

Using the relative cohomology morphisms, we derive numerical conditions for generic unbalancedness of normal bundles for smooth projective toric varieties, answering Question~\hyperref[que:general-splitting-type]{(\ref*{que:general-splitting-type})}. First, we recall the relation between unbalancedness of normal bundles of the rational curves in the family and failure of maximal rank of the morphism \(\Theta_{L_c}^1\) on \(U\).
\begin{lemma}
\label{lem:rank-deficiency-implies-unbalancedness}
Let \(q \in U\). If there exists $c$ such that $\Theta_{L_c}^1|_q$ fails to have maximal rank, then $N_{q}$ is unbalanced.
\end{lemma}

\begin{proof}
By assumption, both the kernel and the cokernel of $\left.\Theta_{L_c}^1\right|_q$ are nonzero.
Sheaf cohomology is given by the right derived functors of global sections, so
restricting the normal resolution \eqref{eq:two-term-resolution-N} to $\mathbb P^1\times\{q\}$ and applying the long exact cohomology sequence after twisting by $L_c$ induces natural surjective and isomorphic maps:
\begin{align*}
H^0(\PP^1,N_{q}\otimes L_c)\twoheadrightarrow\ker(\left.\Theta_{L_c}^1\right|_q), \quad
H^1(\PP^1,N_{q}\otimes L_c)\cong\operatorname{coker}(\left.\Theta_{L_c}^1\right|_q).
\end{align*}
The cohomology groups of \(N_q \otimes L_c\) are therefore nonzero.
The nonvanishing of $H^0$ requires a summand of $N_{f_q}(-c-1)$ of nonnegative degree, while that of $H^1$ requires one of degree at most $-2$.
Thus the largest and smallest summand degrees differ by at least two, so $N_{q}$ is unbalanced.
\end{proof}

To obtain explicit conditions for failure of maximal rank, we first compute the source rank minus the target rank, denoted by \(\delta_{\mathrm{exp}}(c)\), in Lemma~\ref{lem:expected-rank-difference}. Maximal rank requires the kernel to have dimension \(\max\{\delta_{\mathrm{exp}}(c),0\}\). We then define \(\delta_W(c)\), which is shown to be a lower bound for \(\dim\ker(\Theta_{L_c}^1|_q)\) at every \(q\in U\) and induces a sufficient condition for unbalancedness as Theorem~\ref{thm:toric-unbalanced-criterion}.

\begin{lemma}
\label{lem:expected-rank-difference}
Let $c\ge2$.
The rank of the source of $\Theta_{L_c}^1$ minus the rank of its target is given by
\[
\delta_{\mathrm{exp}}(c)=(m+1)c-2-\sum_{\rho\in\Sigma(1)}\max\{c-d_\rho,0\}.
\]
\end{lemma}

\begin{proof}
For \(q\in U\), restrict the short exact sequence~\eqref{eq:two-term-resolution-N} of Theorem~\ref{thm:two-term-resolution-N} to \(\PP^1\times\{q\}\) and twist by \(L_c\). The source and target of \(\left.\Theta_{L_c}^1\right|_q\) are \(H^1(\PP^1,\mathcal E_q\otimes L_c)\) and \(H^1(\PP^1,\mathcal F_q\otimes L_c)\), respectively. Since \(c\ge2\), the decompositions of \(\mathcal E\) and  \(\mathcal F\)  in  Theorem~\ref{thm:two-term-resolution-N} give \(h^1(\PP^1,\mathcal E_q\otimes L_c) = (m+1)c-2\) and   \( h^1(\PP^1,\mathcal F_q\otimes L_c)=\sum_{\rho\in\Sigma(1)}\max\{c-d_\rho,0\}, \) using \(h^1(\PP^1,\mathcal O(a))=\max\{-a-1,0\}\).
\end{proof}

\begin{definition}
For a subspace $W\subseteq V$, set 
\[I(W)=\{\rho\in\Sigma(1)\mid\lambda([D_\rho])\ne0\text{ for some }\lambda\in W\}.\] 
Then for each $c\ge2$, define 
\[
\delta_W(c)=c\dim W-\sum_{\rho\in I(W)}\max\{c-d_\rho,0\}.
\]
\end{definition}

\begin{theorem}
\label{thm:toric-unbalanced-criterion}
If there exist $c\ge2$ and a subspace $W\subseteq V$ such that
\begin{equation}
\label{eq:toric-unbalanced-criterion}
\delta_W(c)>\max\{\delta_{\mathrm{exp}}(c),0\},
\end{equation}
then the morphism $\Theta_{L_c}^1$ fails to have maximal rank at every point of $U$.
\end{theorem}

\begin{proof}
    
Choose $c\ge2$ and $W\subseteq V$ satisfying \eqref{eq:toric-unbalanced-criterion}.
The presentation of \(\mathcal E\) in Theorem~\ref{thm:two-term-resolution-N} induces an inclusion \(j:V\otimes\mathcal O\hookrightarrow\mathcal E\) with cokernel \(\mathcal O(1,0)^{\oplus2}/\operatorname{im}(-s,-t)\cong\mathcal O(2,0)\). Together with the map \(\epsilon_F\), we obtain the following commutative diagram of exact sequences on \(\mathbb P^1\times U\):
\begin{equation}
\label{diag:toric-euler-normal-comparison}
\begin{array}{ccccccccc}
&&0&&0&&0\\
&&\big\downarrow&&\big\downarrow&&\big\downarrow\\
0&\longrightarrow&V\otimes\mathcal O
&\xrightarrow{\mathrm{id}}&V\otimes\mathcal O
&\longrightarrow&0&\longrightarrow&0\\
&&\big\downarrow j&&\big\downarrow\epsilon_F&&\big\downarrow\\
0&\longrightarrow&\mathcal E
&\xrightarrow{\Psi}&\mathcal F
&\longrightarrow&\mathcal N&\longrightarrow&0\\
&&\big\downarrow&&\big\downarrow&&\big\Vert\\
0&\longrightarrow&T_{\mathbb P^1\times U/U}
&\xrightarrow{dF}&F^*T_X
&\longrightarrow&\mathcal N&\longrightarrow&0\\
&&\big\downarrow&&\big\downarrow&&\big\downarrow\\
&&0&&0&&0.
\end{array}
\end{equation}
Twisting by $\pi_1^*L_c$ and applying $R^1\pi_{2*}$, denote the following induced morphisms
\[
\begin{aligned}
A_W&:=R^1(\pi_2)_*(\epsilon_F|_W\otimes\mathrm{id}_{\pi_1^*L_c}),\quad
Y_W:=R^1(\pi_2)_*(j|_W\otimes\mathrm{id}_{\pi_1^*L_c}).
\end{aligned}
\]
Both maps have source $W\otimes H^1(\PP^1,L_c)\otimes\mathcal O_U$, and commutativity gives $\Theta_{L_c}^1\circ Y_W=A_W$.
Since \(c\ge2\), \(H^0(\PP^1,\mathcal O(1-c))=0\); thus twisting the first column of \eqref{diag:toric-euler-normal-comparison} by \(\pi_1^*L_c\) and pushing forward gives
\begin{equation}
\label{eq:toric-pushforward-extension}
\begin{aligned}
0\longrightarrow V\otimes H^1(\PP^1,L_c)\otimes\mathcal O_U
&\xrightarrow{Y_V} R^1(\pi_2)_*(\mathcal E\otimes\pi_1^*L_c)
\longrightarrow H^1(\PP^1,\mathcal O(1-c))\otimes\mathcal O_U\longrightarrow0.
\end{aligned}
\end{equation}
This exact sequence of vector bundles shows that \(Y_W\) is an inclusion of a subbundle, so restricting \(Y_W\) to \(\ker A_W\) gives an injection of coherent sheaves on \(U\),
\begin{equation}
\label{eq:toric-kernel-inclusion}
\ker A_W\hookrightarrow\ker\Theta_{L_c}^1.
\end{equation}
Since $Y_W$ remains injective on every fiber, the identity $\Theta_{L_c}^1\circ Y_W=A_W$ also gives the corresponding injection between the kernels of the fiber maps.
Now, the $\rho$-component of $\epsilon_F|_W$ vanishes for $\rho\notin I(W)$, so $A_W$ factors through the vector bundle
\[
\bigoplus_{\rho\in I(W)}H^1(\PP^1,\mathcal O(d_\rho)\otimes L_c)\otimes\mathcal O_U(\zeta_\rho).
\]
Comparing its rank with the source rank $c\dim W$ gives, for every $q\in U$,
\[
\dim\ker(\left.\Theta_{L_c}^1\right|_q)\ge\dim\ker(\left.A_W\right|_q)
\ge c\dim W-\sum_{\rho\in I(W)}\max\{c-d_\rho,0\}=\delta_W(c).
\]
Together with inequality \eqref{eq:toric-unbalanced-criterion}, this shows that the kernel dimension exceeds the value $\max\{\delta_{\mathrm{exp}}(c),0\}$ required for maximal rank.
Thus $\Theta_{L_c}^1$ fails to have maximal rank at every point of $U$.
\end{proof}

We show that for toric varieties of Picard rank two, under additional hypotheses, it suffices to test the criterion in Theorem~\ref{thm:toric-unbalanced-criterion} on a single geometrically induced subspace \(W_\pi\), defined as follows.
\begin{definition}
\label{def:morphism-subspace}
Let $\pi\colon X\to Y$ be an equivariant morphism between smooth projective toric varieties.
Define $W_\pi$  to be the annihilator \(\{\lambda\in V\mid\lambda\circ\pi^*=0\}\) of the image of the pullback map $\pi^*\colon\operatorname{Pic}(Y)\to\operatorname{Pic}(X)$.
\end{definition}

\begin{proposition}
\label{prop:optimal-Wpi-projective-bundle}
Assume that $X$ has Picard rank $2$, $d_\rho\ge0$ for every $\rho\in\Sigma(1)$, and that $X$ contains a rigid prime divisor $D$, that is, $h^0(X,\cO_X(D))=1$.
Let $A$ be a primitive nef non-big divisor class.
Let $\pi : X\to\PP^b$ be the morphism induced by the complete linear series for \(A\).
For each $c\ge2$, if there exists a subspace $W\subseteq V$ satisfying \eqref{eq:toric-unbalanced-criterion}, then $W_\pi$ is the only subspace satisfying that inequality and is the unique maximizer of $\delta_W(c)$ among all subspaces of $V$.
In particular, to verify the criterion in Theorem~\ref{thm:toric-unbalanced-criterion}, it suffices to test \eqref{eq:toric-unbalanced-criterion} with $W=W_\pi$.
\end{proposition}

\begin{proof}
Because the torus translates are linearly equivalent to each other, rigidity implies that $D$ is a torus-invariant prime divisor associated with a ray in \(\Sigma(1)\).
Peter Kleinschmidt's rank-two classification~\cite[Theorem~1]{Kleinschmidt1988} identifies $\pi$ as the projection of a split projective bundle.
The nonzero nef class $A$ is base-point-free by Cox--Little--Schenck~\cite[Theorem~6.3.12]{CoxLittleSchenck2011}, so $h^0(X,\cO_X(A)) = b + 1 \ge2$; since $A$ is primitive, rigidity prevents $D$ from being proportional to $A$.
Thus $D$ dominates the base and, being torus-invariant, restricts to a coordinate hyperplane on a fiber $\pi^{-1}(y)\cong\PP^{n-b}$ over the dense torus of $\PP^b$.
Restriction of line bundles on $X$ to this fiber, $\mathcal L\mapsto\mathcal L|_{\pi^{-1}(y)}$, induces the exact sequence
\[
0\longrightarrow\mathbb ZA\longrightarrow\operatorname{Pic}(X)
\xrightarrow{\mathrm{res}}\operatorname{Pic}(\pi^{-1}(y))\cong\mathbb Z\longrightarrow0.
\]
The class of $\cO_X(D)$ maps to the hyperplane generator $1$, so $\operatorname{Pic}(X)=\mathbb ZA\oplus\mathbb ZD$. 
Note that $W_\pi=\ker\operatorname{ev}_A$ is a line generated by  $D^\vee$. The invariant prime divisors comprise $b+1$ divisors of class $A$, the divisor $D$, and remaining divisors of classes $D+jA$ for integers $j > 0$. 

Fix $c\ge2$ and denote  \( h=A\cdot\beta\)   and   \(z=D\cdot\beta\).
Then $h\ge0$ by the nefness of $A$ and the effectivity of $\beta$, while $z\ge0$ by the assumption $d_\rho\ge0$ for all \(\rho\).
The zero subspace cannot satisfy \eqref{eq:toric-unbalanced-criterion} because \(\delta_{0}(c) \leq 0\), and neither can $V$ because
\[
\delta_V(c)=\delta_{\mathrm{exp}}(c)-c+2\le\delta_{\mathrm{exp}}(c).
\]
Thus any subspace \(W\) satisfying the inequality is a line in \(V\). We show that such a subspace \(W\) must be \(W_\pi\); because \(W_\pi\) is a line, it suffices to show that \(\operatorname{ev}_A|_W=0\).

Let $W$ be a line satisfying \eqref{eq:toric-unbalanced-criterion}; suppose for the sake of contradiction that $\operatorname{ev}_A|_W\ne0$. Then all $b+1$ invariant divisors of class $A$ occur in the defining sum for $\delta_W(c)$, giving
\[
0<\delta_W(c)
\le c-(b+1)\max\{c-h,0\}
\le c-2\max\{c-h,0\}
\le2h-c,
\]
where we used $b+1\ge2$.
Hence $c<2h$, so every divisor of class $D+jA$ with $j\ge2$ has degree $z+jh>c$. In particular, in the inequalities
{
\[
\sum_{\rho\in\Sigma(1)\setminus I(W)}\max\{c-d_\rho,0\}>2c-2\ge c,
\]
only divisors of classes $D$ and $D+A$ can contribute positively to this sum.
Since $D$ occurs only once by rigidity and contributes at most $c$, a divisor of class $D+A$ must contribute positively. Hence we must have \(\operatorname{ev}_{D+A}|_W=0\) and \(z + h < c\). 
Since $\operatorname{ev}_A|_W\ne0$ by assumption, we have
\[
\operatorname{ev}_D|_W=-\operatorname{ev}_A|_W\ne0.
\]
It follows that
\[
\delta_W(c) =  c-\sum_{\rho\in I(W)}\max\{c-d_\rho,0\}\le c-(c-h)-(c-z)=h+z-c<0,
\]
}
but this contradicts the hypothesis that \eqref{eq:toric-unbalanced-criterion} requires $\delta_W(c)>0$.
Therefore, $\operatorname{ev}_A|_W=0$, so that $W=W_\pi$.

\end{proof}

Later in Section~\ref{sec:linear-blowup-example}, we apply this Proposition~\ref{prop:optimal-Wpi-projective-bundle} to the blowup of projective space along a linear subspace, which satisfies the conditions in this proposition.

\section{Jumping Loci}
\label{sec:jumping_loci}

Question~\hyperref[que:jumping-locus]{(\ref*{que:jumping-locus})} concerns the jumping loci, where the splitting type of the normal bundle of an unramified rational curve changes as the curve varies in \(U\). In this section, we make the question precise: we introduce two notions of jumping loci, the \emph{splitting strata} defined in Section \ref{subsec:splitting-strata-codimension} and the \emph{cohomology-jump loci} defined in Section \ref{subsec:jumping-divisor-class-general}. We show that the two notions are related, and the latter is particularly useful being defined by the rank conditions of the relative cohomology morphisms. We conclude with a Porteous-type argument for the expected class of the jumping loci in the Chow ring of \(U\).

\subsection{Splitting strata}
\label{subsec:splitting-strata-codimension}

\begin{lemmadefinition}
Let \(U\) be the parameter space of unramified rational curves in a smooth projective toric variety \(X\) of dimension \(n\ge2\) with curve class \(\beta\).  Then the normal bundle has degree \(D=\beta\cdot(-K_X)-2=\sum_{\rho\in\Sigma(1)}d_\rho-2\) and the \emph{balanced splitting type }\(\mathbf{a}^\circ=(a_1^\circ,\ldots,a_{n-1}^\circ)\) is
\begin{equation}
\label{eq:general-splitting-type}
a_i^\circ= \begin{cases}c_*, & \text{for } 1 \leq i \leq n-1-b, \\ 
    c_*+1, & \text{for } n-b \leq i\leq n-1, \end{cases}
\end{equation}
where \(c_* =  \left\lfloor \frac{D}{n-1} \right\rfloor\) and \(b\) is the remainder in the Euclidean division \(D = (n-1)c_* + b\).
\end{lemmadefinition}

\begin{proof}
For \(q\in U\), the normal exact sequence gives \(D=\deg N_q=\deg f_q^*T_X-\deg T_{\PP^1}=\beta\cdot(-K_X)-2\). Alternatively, restricting the two-term resolution in Theorem~\ref{thm:two-term-resolution-N} to \(\PP^1\times\{q\}\), the bundles \(\cE\) and \(\cF\) have degrees \(2\) and \(\sum_{\rho\in\Sigma(1)}d_\rho\), respectively, so \(D=\deg N_q=\sum_{\rho\in\Sigma(1)}d_\rho-2\).
\end{proof}

\begin{definition}
\label{def:normal-jumping-locus}
For an integer tuple \(\boldsymbol\delta=(\delta_1,\ldots,\delta_{n-1})\) satisfying \(\sum_i\delta_i=0\) and \(a_1^\circ+\delta_1\le\cdots\le a_{n-1}^\circ+\delta_{n-1}\), the \emph{splitting stratum with unbalancedness \(\boldsymbol\delta\)} is
\[
\cJ_{\boldsymbol\delta}=\left\{q\in U\mid N_q \text{ has splitting type }\mathbf{a}^\circ+ \boldsymbol{\delta}\right\}.
\]
Thus \(\boldsymbol\delta=\mathbf a-\mathbf a^\circ\) records deviation from the balanced splitting type, and \(\cJ_{\mathbf0}\) is the balanced locus. Write \(\mathbf a^\circ+\boldsymbol\delta^{\mathrm{gen}}\) for the splitting type of the general normal bundle in \(U\). The \emph{first jumping locus} is the  union \(\cJ=\bigsqcup_{\boldsymbol\delta\ne\boldsymbol\delta^{\mathrm{gen}}}\cJ_{\boldsymbol\delta}\). When the general normal bundle is balanced, \(\boldsymbol\delta^{\mathrm{gen}}=\mathbf0\), and the first jumping locus coincides with the unbalanced locus.
\end{definition}

\begin{remark}
\label{rmk:deformation-expected-codimension}
Let \(\boldsymbol\delta=(\delta_1,\ldots,\delta_{n-1})\) be as in Definition~\ref{def:normal-jumping-locus}, and set \(E_{\boldsymbol\delta}=\bigoplus_{i=1}^{n-1}\cO_{\PP^1}(a_i^\circ+\delta_i)\). By Larson's formula~\cite{Larson}, the splitting stratum \(\cJ_{\boldsymbol\delta}\) has expected codimension
\[
u(E_{\boldsymbol\delta})=h^1(\PP^1,\operatorname{End}E_{\boldsymbol\delta})
=\sum_{i<j}\max\{(a_j^\circ-a_i^\circ)+(\delta_j-\delta_i)-1,0\}
\] 
in the sense that, for every \(q\in\cJ_{\boldsymbol\delta}\), the local codimension of \(\cJ_{\boldsymbol\delta}\) at \(q\) is at most \(u(E_{\boldsymbol\delta})\).
\end{remark}

\subsection{Cohomological jumping loci}
\label{subsec:jumping-divisor-class-general}

The relative cohomology morphisms allow us to describe the splitting strata by their rank conditions. In this section, we first define cohomology-jump loci for fixed twists, induced by the relative cohomology morphisms, and relate them to splitting strata by considering all twists. We then use the determinantal equations of the relative cohomology morphisms to compute expected codimensions and classes of the jumping loci via Porteous's formula.

For the remainder of the paper, \(r_L\) denotes the generic rank of \(\Theta_L^1\) over \(U\). For a fixed twist \(L\) and \(j\ge0\), the locus \(J_L^j\) consists of points \(q\in U\) where \(h^1(N_q\otimes L)\) is at least \(j\) greater than its value at a general point of \(U\).

\begin{definition}
\label{def:cohomology-jump-loci}
For a line bundle \(L\) on \(\PP^1\) and \(j\ge0\), the \emph{\(j\)-th cohomological jump locus} is the locus where \(\Theta_L^1\) has rank at most \(r_L - j\):
\[
J_L^j:=\left\{q\in U\mid r_L - \operatorname{rank}(\left.\Theta_L^1\right|_q)\geq j\right\}.
\]
\end{definition}
For \(1\le j\le r_L\), the cohomology-jump locus \(J_L^j\) carries the determinantal closed scheme structure cut out by the \((r_L-j+1)\times(r_L-j+1)\) minors of a local matrix for \(\Theta_L^1\). More concretely, after choosing monomial bases, \(\Theta_L^1\) is represented by the multiplication matrices of the entries of the applicable matrix for \(\Psi\) in Remark~\ref{rmk:matrix-Psi}. Its matrix entries are linear combinations of the Cox coefficients over \(k\) and are sections of the corresponding \(\cO_U(\zeta_\rho)\). The equations depend on the characteristic of \(k\); we illustrate this construction for the blowup family in Section~\ref{sec:linear-blowup-example}, with explicit matrices and computations of their ranks and determinants in Appendix~\ref{app:theta-q-computation}.  

Applying the cohomological characterization of splitting types in \cite[Section~2]{Larson} across all twists, the following lemma expresses each splitting stratum as an intersection of exact-rank loci.

\begin{lemma}
\label{lem:jumping-generic-rank}
For \(\boldsymbol\delta\) as in Definition~\ref{def:normal-jumping-locus}, let
\[
j_{\boldsymbol\delta}(c)
=\sum_{i=1}^{n-1}\left(
h^1\!\left(\cO(a_i^\circ+\delta_i)\otimes L_c\right)
-h^1\!\left(\cO(a_i^\circ+\delta_i^{\mathrm{gen}})\otimes L_c\right)
\right),
\]
the change in \(h^1\) after twisting by \(L_c\) prescribed by the splitting type \(\mathbf a^\circ+\boldsymbol\delta\), relative to the general splitting type. If \(j_{\boldsymbol\delta}(c)\ge0\) for every \(c\in\Z\), then
\[
\cJ_{\boldsymbol\delta}
=\bigcap_{c\in\Z}
\left(J_{L_c}^{\,j_{\boldsymbol\delta}(c)}
\setminus J_{L_c}^{\,j_{\boldsymbol\delta}(c)+1}\right).
\]
If \(j_{\boldsymbol\delta}(c)<0\) for some \(c\), then \(\cJ_{\boldsymbol\delta}=\varnothing\). In particular, the first jumping locus satisfies
\begin{equation}
\label{eq:jumping-locus-union}
\cJ=\bigcup_{c\in\Z}J_{L_c}^1.
\end{equation}
\end{lemma}

\begin{proof}
By the fiberwise cohomology sequence in Lemma--Definition~\ref{thm:universal-pushforward}, \(H^1(N_q\otimes L_c)\cong\operatorname{coker}(\left.\Theta_{L_c}^1\right|_q)\). Comparing its dimension with that of the general fiber gives
\[
h^1(N_q\otimes L_c)-\sum_{i=1}^{n-1}h^1(\cO(a_i^\circ+\delta_i^{\mathrm{gen}})\otimes L_c)
=r_{L_c}-\operatorname{rank}(\left.\Theta_{L_c}^1\right|_q).
\]
Two vector bundles on \(\PP^1\) have the same splitting type if and only if their \(h^1\) dimensions agree after every twist. Thus \(N_q\) has splitting type \(\mathbf a^\circ+\boldsymbol\delta\) if and only if \(r_{L_c}-\operatorname{rank}(\left.\Theta_{L_c}^1\right|_q)=j_{\boldsymbol\delta}(c)\) for every \(c\). These equalities give the stated intersection formula.

Since \(\operatorname{rank}(\left.\Theta_{L_c}^1\right|_q)\le r_{L_c}\) at every point, \(j_{\boldsymbol\delta}(c)<0\) for some \(c\) implies \(\cJ_{\boldsymbol\delta}=\varnothing\). By the same characterization, the splitting type differs from the general type if and only if \(\operatorname{rank}(\left.\Theta_{L_c}^1\right|_q)<r_{L_c}\) for some \(c\), hence the equality~\eqref{eq:jumping-locus-union} holds.
\end{proof}

The description of \(J_{L_c}^j\) by minors following Definition~\ref{def:cohomology-jump-loci} leads to the following application of Porteous's formula in Proposition~\ref{prop:cohomology-jump-class}. The matrix dimensions and the sizes of these minors determine the expected codimension, while the Chern classes of the source and target of \(\Theta_{L_c}^1\) determine the expected class. The latter equals the fundamental class when the locus has the expected codimension.

\begin{proposition}
\label{prop:cohomology-jump-class}
Fix \(c\ge2\), and let \(e\) and \(f\) be the ranks of the source and target of \(\Theta_{L_c}^1\), respectively. Set
\[
C(z)=\prod_{\rho\in\Sigma(1)}(1+\zeta_\rho z)^{\max\{c-d_\rho,0\}}\in A^*(U)[z],
\]
where \(A^*(U)\) is the Chow ring of \(U\) and \(z\) is a formal variable.
For \(1\le j\le r_{L_c}\), the locus \(J_{L_c}^j\) has determinantal expected codimension
\[
\kappa_j=(e-r_{L_c}+j)(f-r_{L_c}+j).
\]
If \(J_{L_c}^j\) is nonempty and has pure codimension \(\kappa_j\), then its class is given by
\begin{equation}
\label{eq:porteous-cohomology-jump-class}
[J_{L_c}^j]
=\det\!\left([C(z)]_{f-r_{L_c}+j+a-b}\right)_{1\le a,b\le e-r_{L_c}+j}
\quad\text{in }A^{\kappa_j}(U),
\end{equation}
where \([P(z)]_\ell\) denotes the coefficient of \(z^\ell\), with negative-index coefficients taken to be zero.

\end{proposition}

\begin{proof}
Denote the source and target of \(\Theta_{L_c}^1\) by
\[
E_{L_c}=R^1\pi_{2*}(\cE\otimes\pi_1^*L_c),\qquad
F_{L_c}=R^1\pi_{2*}(\cF\otimes\pi_1^*L_c).
\]
By Lemma--Definition~\ref{thm:universal-pushforward},
\(
E_{L_c}\cong\cO_U^{\oplus e}\) and \(
F_{L_c}\cong\bigoplus_{\rho\in\Sigma(1)}
\cO_U(\zeta_\rho)^{\oplus\max\{c-d_\rho,0\}}.
\)
The formal expansion \(C(z)\) is the Chern polynomial of \(F_{L_c}-E_{L_c}\). By Porteous's formula \cite{Fulton1998}, the expected codimension of \(J_{L_c}^j\) is \(\kappa_j=(e-r_{L_c}+j)(f-r_{L_c}+j)\), and when the locus has the expected codimension, then its class is given by \eqref{eq:porteous-cohomology-jump-class}.
\end{proof}

\begin{corollary}
\label{cor:first-cohomology-jump-class}
With the notation of Proposition~\ref{prop:cohomology-jump-class}, if \(r_{L_c}=\min\{e,f\}>0\), the first cohomological jump locus \(J_{L_c}^1\) has expected codimension \(\kappa=|f-e|+1\). If \(J_{L_c}^1\) is nonempty and has pure codimension \(\kappa\), then its class is given by
\begin{equation}
\label{eq:porteous-first-jump-class}
[J_{L_c}^1]=
\begin{cases}
[C(z)]_\kappa,&e\le f,\\[2pt]
\bigl[C(-z)^{-1}\bigr]_\kappa,&e>f,
\end{cases}
\quad\text{in }A^\kappa(U),
\end{equation}
where the inverse \(C(-z)^{-1}\) is expanded as a formal power series.
\end{corollary}

\begin{proof}
Use the notation \(E_{L_c}\) and \(F_{L_c}\) from the proof of Proposition~\ref{prop:cohomology-jump-class}. For \(j=1\), generic maximal rank gives \(\kappa=|f-e|+1\). If \(e\le f\), the determinant has size one and equals \(c_\kappa(F_{L_c}-E_{L_c})=[C(z)]_\kappa\). If \(e>f\), apply the same formula to \(F_{L_c}^\vee\to E_{L_c}^\vee\), whose maximal minors define the same scheme. The resulting class is \(c_\kappa(E_{L_c}^\vee-F_{L_c}^\vee)=[C(-z)^{-1}]_\kappa\), giving \eqref{eq:porteous-first-jump-class}.
\end{proof}

\begin{remark}
The generic rank \(r_{L_c}\) can also be read from the general splitting type \(\boldsymbol a^\circ + \boldsymbol\delta^{\mathrm{gen}}\):
\[
r_{L_c}=f-\sum_{i=1}^{n-1}
\max\{c-a_i^\circ-\delta_i^{\mathrm{gen}},0\}.
\]
Indeed, the sum is the dimension of the cokernel at a general point, by Lemma~\ref{lem:jumping-generic-rank}. Proposition~\ref{prop:cohomology-jump-class} computes classes of the individual cohomological jump loci \(J_{L_c}^j\). Identifying one of these loci with the first jumping locus \(\cJ\) requires an additional argument, as in the example of Section~\ref{sec:linear-blowup-example}.
\end{remark}

\section{Application to blowup of projective space}
\label{sec:linear-blowup-example}

As an application, we apply the construction to the case when \(X\) is  the blowup  \(X_{r,s}=\operatorname{Bl}_{\Lambda}\PP^r\) of projective space along a linear subspace \(\Lambda\cong\PP^s\). Denote the pullback to \(X\) of the hyperplane class on \(\PP^r\) by \(H\), and the exceptional divisor by \(E\). We fix the curve class \(\beta\) of degree-\(d\) unramified rational curves in \(X\) with intersection number \(k\) with \(E\). Equivalently, if \(\beta=f_*[\PP^1]\), then \(H\cdot\beta=d\) and \(E\cdot\beta=k\); we write \(\beta = (d,k)\). We assume \(r\geq2\), \(0\leq s\leq r-2\), and \(d>k\geq0\).

\begin{remark}
Under these assumptions, \(U\ne\varnothing\) in every characteristic: a general morphism \(\PP^1\to X_{r,s}\) of class \((d,k)\) is unramified by Cela--Lian~\cite[Proposition~2.1.1]{CelaLian2026}.
\end{remark}


\subsection{Parameter space}
\label{sec:blowup-parameter-space}
\label{app:cox-coordinate-linear-blowup-example}
We first construct the parameter space \(U\) following the procedure in Section \ref{sec:parameter-space-preliminaries}.
Choose homogeneous coordinates \([x_0:\cdots:x_r]\) on \(\PP^r\) so that the blowup center is of the form \(\Lambda=V(x_{s+1},\ldots,x_r)\). 
The primitive ray generators of the fan of \(X\) are
\[  
  \rho_0=-\sum_{j=1}^r e_j,\qquad \rho_i=e_i\quad(1\le i\le r), \quad \rho_E = e_{s+1}+\cdots+e_r,
\]
where \(e_1,\ldots,e_r\) form the standard lattice basis.
Denote the Cox variables associated with \(\rho_0,\ldots,\rho_r,\rho_E\), in that order, by
\((u_0,\ldots,u_s,v_{s+1},\ldots,v_r,w)\). Thus the Cox ring is
\[
  R_\Sigma=k[u_0,\ldots,u_s,v_{s+1},\ldots,v_r,w],
\]
where the Cox variables have degrees \(\deg(u_i)=H\), \(\deg(v_j)=H-E,\) and \( \deg(w)=E.\)
Denote \(\mathbf u=(u_0,\ldots,u_s)\) and \(\mathbf v=(v_{s+1},\ldots,v_r)\).
The blowup map \(X \to \PP^r\) is a toric-equivariant map, given in Cox coordinates by
\[
  (\mathbf u,\mathbf v,w)\longmapsto
  [u_0:\cdots:u_s:wv_{s+1}:\cdots:wv_r].
\]

We now use the same symbols for homogeneous polynomials on the fixed source \(\PP^1\). Set \(S_m=H^0(\PP^1_k,\cO_{\PP^1_k}(m))\), so that \(u_i\in S_d\), \(v_j\in S_{d-k}\), and \(w\in S_k\). For the class \(\beta=(d,k)\), the coefficient space \(\cA\) is
\[
  \cA = \cA^{r,s}_{d,k}
  =S_d^{\oplus(s+1)}
    \oplus S_{d-k}^{\oplus(r-s)}
    \oplus S_k.
\]
The stable locus \(\cA^\circ\) consists of tuples with \(\mathbf v\not\equiv0\) and \((\mathbf u,w)\not\equiv0\).
For an honest morphism, the base-point-free conditions are
\[
  (v_{s+1}(p),\ldots,v_r(p))\ne0,\quad
  (u_0(p),\ldots,u_s(p),w(p))\ne0, \quad  \text{ for all }p \in \PP^1(k).
\]
The Cox-torus quotient of \(\cA^\circ\) is the space  \(Q\) of stable toric quasimaps introduced in Section~\ref{sec:parameter-space-preliminaries}; \(Q\) can be realized as a projective bundle as follows.
\begin{lemma}[{\cite{CelaLian2026}}]
Assume \(r\geq2\), \(0\leq s\leq r-2\), and \(d>k\geq0\). Then \(Q\) is the projective bundle over \(P_v = \PP\bigl(S_{d-k}^{\oplus(r-s)}\bigr)\),
\[ p : Q \cong \PP_{P_v}(\cE^{r,s}_{d,k})\longrightarrow P_v,\]
where \(\cE^{r,s}_{d,k}\) is the vector bundle
\[\cE^{r,s}_{d,k} =\bigl(S_d^{\oplus(s+1)}\otimes_k\cO_{P_v}\bigr)
    \oplus\bigl(\cO_{P_v}(-1)\otimes_k S_k\bigr).\]
With projectivization parametrizing lines, the Cox-torus quotient map \(\cA^\circ\to Q\) is given by
\[
  (\mathbf u,\mathbf v,w)
  \longmapsto\bigl([\mathbf v],[\mathbf u;\mathbf v\otimes w]\bigr).
\]
\end{lemma}

Let \(\zeta_u=[\cO_{Q}(1)]\) and \(\zeta_v=[p^*\cO_{P_v}(1)]\) be the tautological and base hyperplane classes in \(\operatorname{Pic}(Q)\).  We use the same symbols for their first Chern classes and their restrictions to \(U\). By the projective-bundle description, the Picard group of \(Q\) is freely generated by \(\zeta_u\) and \(\zeta_v\); for simplicity, denote \(\cO_{Q}(m\zeta_u+n\zeta_v):=\cO_{Q}(m)\otimes p^*\cO_{P_v}(n)\) for integers \(m,n\).

\subsection{Relative normal resolution and cohomology morphisms}

We now derive the two-term resolution of the relative normal bundle \(\cN\) of the universal curve \(F : \PP^1 \times U \to X_{r,s}\). Globally on \(\PP^1 \times Q\), the Cox coordinates correspond to the following universal sections:
\[
\begin{aligned}
  {u_0,\cdots,u_s} &\in H^0\bigl(\PP^1\times Q,
    \cO(d,\zeta_u) \bigr),\\
  {v_{s+1},\cdots,v_r} &\in H^0\bigl(\PP^1\times Q,
    \cO(d-k,\zeta_v)\bigr),\\
  w &\in H^0\bigl(\PP^1\times Q,
    {\cO(k,\zeta_u-\zeta_v)}\bigr).
\end{aligned}
\]
Then Theorem \ref{thm:two-term-resolution-N} applied to \(X_{r,s}\) reduces to the following Corollary.
\begin{corollary}
  \label{cor:two-term-resolution-blowup} For an integral class \(\beta=(d,k)\), the relative normal bundle has a resolution
\[
0\longrightarrow\cE\xrightarrow{\Psi}\cF
\longrightarrow\cN\longrightarrow0\quad\text{on }\PP^1\times U,
\]
with
\[
\cF=\cO(d,\zeta_u)^{\oplus(s+1)}
\oplus\cO(d-k,\zeta_v)^{\oplus(r-s)}
\oplus\cO(k,\zeta_u-\zeta_v).
\]
In the row order \((u_0,\ldots,u_s,v_{s+1},\ldots,v_r,w)\), denote  
\[
A_H= \epsilon_F(H^\vee\otimes1) = (u_i,v_j,0)^T,\qquad A_E = \epsilon_F(E^\vee\otimes1)=(0,-v_j,w)^T.
\]
Then the source \(\cE\) and the matrix \(\Psi\) are given in the following three cases.
\[
\renewcommand{\arraystretch}{1.4}
\begin{array}{|c|c|c|}
\hline
\text{condition in the field \(k\)}&\cE&\text{columns of }\Psi\\ \hline
d\ne0&\cO\oplus\cO(1,0)^{\oplus2}&(-A_E,D_s,D_t)\\
d=0,\ k\ne0&\cO\oplus\cO(1,0)^{\oplus2}&(A_H,D_s,D_t)\\
d=k=0&\cO^{\oplus2}\oplus\cO(2,0)&(A_H,A_E,(H_\rho)_\rho)\\ \hline
\end{array}
\]
In the third case, the section \(H_\rho  \in H^0(\PP^1 \times U,\cO(d_\rho - 2, \zeta_\rho)) \) is given  explicitly as follows. Expand the Cox polynomial \(f_\rho \in \{u_0,\ldots,u_s,v_{s+1},\ldots,v_r,w\}\) with respect to monomials in \(s,t\) as \(f_\rho=\sum_{\nu=0}^{d_\rho}f_\nu s^{d_\rho-\nu}t^\nu\). Then \(H_\rho \) is given by
\begin{equation} 
  \label{eq:H_f-blowup}
H_\rho =
\begin{cases}
-\displaystyle\sum_{\nu=0}^{d_\rho-2}(\nu+1)f_{\nu+1}s^{d_\rho-2-\nu}t^\nu,&d_\rho\ge2,\\
0,&d_\rho=0.
\end{cases}
\end{equation}
\end{corollary}

\begin{proof}
We apply Theorem~\ref{thm:two-term-resolution-N} to \(X=X_{r,s}\). Here \(m=2\), and in the dual basis to \(H,E\), we have \(\bar\beta=dH^\vee+kE^\vee\). The direct sum defining \(\cF\) has \(s+1\) summands \(\cO(d,\zeta_u)\) indexed by \(u_0,\ldots,u_s\), \(r-s\) summands \(\cO(d-k,\zeta_v)\) indexed by \(v_{s+1},\ldots,v_r\), and one summand \(\cO(k,\zeta_u-\zeta_v)\) indexed by \(w\). Therefore, \(\cF\) is given as above.

To identify the source \(\cE\) and the morphism \(\Psi\), recall that \(\Psi\) is induced by \(\widetilde{\Psi} = \epsilon_F + J\). Evaluating on the Cox divisor classes in the specified row order, \(\epsilon_F\) acts as
\[
\epsilon_F(aH^\vee+bE^\vee)
=\bigl(a u_i,(a-b)v_j,bw\bigr)^T,
\]
so the images of \(H^\vee,E^\vee\) are exactly \(A_H,A_E\). Similarly, the Jacobian \(J = (D_s, D_t)\) is given by
\[
D_s=(\partial_su_i,\partial_sv_j,\partial_sw)^T,
\qquad
D_t=(\partial_tu_i,\partial_tv_j,\partial_tw)^T.
\]

If \(d\ne0\) in the base field, then \(-E^\vee\) spans a complement to \(k\bar\beta\) in \(V\); Remark~\ref{rmk:matrix-Psi} therefore gives \(\cE\cong\cO\oplus\cO(1,0)^{\oplus2}\) with columns \((-A_E,D_s,D_t)\). If \(d=0\) but \(k\ne0\), choose the complement spanned by \(H^\vee\) instead, giving the same source and columns \((A_H,D_s,D_t)\).

If \(d=k=0\) in the base field, then \(\bar\beta=0\), so Theorem~\ref{thm:two-term-resolution-N} gives \(\cE\cong\cO^{\oplus2}\oplus\cO(2,0)\). Using \(H^\vee,E^\vee\) as the basis of \(V\), Remark~\ref{rmk:matrix-Psi} gives columns \(A_H,A_E,(H_\rho)_\rho\). Since every \(d_\rho\in\{d,d-k,k\}\) vanishes in the base field, differentiation of the monomial expansion shows that \eqref{eq:H_f-blowup} satisfies
\[
\partial_sf_\rho=tH_\rho,
\qquad \partial_tf_\rho=-sH_\rho.
\]
It is therefore precisely the column induced by the Jacobian on \(\cO(2,0)\) in that remark.
\end{proof}



Next, we identify the relative cohomology morphisms. Twist the relative normal sequence by \(\pi_1^*L_c\). Then
\[
\begin{aligned}
\cE(-c-1,0)&\cong
\begin{cases}
\cO(-c-1,0)\oplus\cO(-c,0)^{\oplus2},&(d,k)\ne(0,0)\text{ in }k^2,\\
\cO(-c-1,0)^{\oplus2}\oplus\cO(1-c,0),&(d,k)=(0,0)\text{ in }k^2,
\end{cases}\\
\cF(-c-1,0)&=\cO(d-c-1,\zeta_u)^{\oplus(s+1)}
\oplus\cO(d-k-c-1,\zeta_v)^{\oplus(r-s)}
\oplus\cO(k-c-1,\zeta_u-\zeta_v).
\end{aligned}
\]
 The map uses the corresponding Euler columns from Corollary~\ref{cor:two-term-resolution-blowup}, including \(A_H\) when \(d=0\) but \(k\ne0\) in the field.
For \(i=0,1\), the induced morphisms \(\Theta_{L_c}^i\) in Lemma--Definition~\ref{thm:universal-pushforward} take the form
\[
  \Theta_{L_c}^i\colon
  R^i(\pi_2)_*\cE(-c-1,0)
  \longrightarrow
  R^i(\pi_2)_*\cF(-c-1,0).
\]

\subsection{General unbalancedness}
\label{subsec:blowup-unbalanced-criterion}

We now apply Theorem~\ref{thm:toric-unbalanced-criterion} to obtain a sufficient criterion for unbalancedness of the general normal bundle. Proposition~\ref{prop:blowup-cela-lian-equivalence} below establishes the equivalence of our numerical criterion with that of Cela--Lian \cite[Theorem~1.4.2]{CelaLian2026}. In particular, under the additional hypotheses \(p\ne2\), \(d\ge r\), and \(d-k\ge r-s-1\),  \cite[Theorem~1.4.2]{CelaLian2026} shows that the condition in Theorem~\ref{thm:toric-unbalanced-criterion} is sharp: it is equivalent to unbalancedness of the general normal bundle.

Take \(A = H-E\) and \(D = E\); then they satisfy the conditions for Proposition~\ref{prop:optimal-Wpi-projective-bundle}, and the morphism \(\pi : X \to \PP^{r-s-1}\) is the morphism induced by \(|H-E|\).
The subspace \(W_\pi\) is given by
 \[W_\pi = \ker \ev_{H-E} \cong  k(H^\vee+E^\vee).\]
By  Proposition~\ref{prop:optimal-Wpi-projective-bundle},  it suffices to check Condition \eqref{eq:toric-unbalanced-criterion} with $W=W_\pi$.

\begin{proposition}
\label{prop:blowup-cela-lian-equivalence}
Let $r,s,d,k$ be integers with $0\le s\le r-2$ and $d>k\ge0$, and let $W_\pi$ be as above.
The following conditions are equivalent:
\begin{enumerate}[(i)]
\item There exists $c\ge2$ such that
\[
\delta_{W_\pi}(c)>\max\{\delta_{\mathrm{exp}}(c),0\}.
\]
\item We have $s<r-2$ and
\[
2d<(r-s)(k-1)+4+(r-s-2)\left\lceil\frac{k}{s+1}\right\rceil.
\]
\end{enumerate}
\end{proposition}


\begin{proof}
    We organize Condition (i) as follows. By definition, we have
\[
\begin{aligned}
\delta_{W_\pi}(c)
&=c-(s+1)\max\{c-d,0\}-\max\{c-k,0\},\\
\delta_{\mathrm{exp}}(c)
&=3c-2-(s+1)\max\{c-d,0\} -(r-s)\max\{c-d+k,0\}-\max\{c-k,0\}.
\end{aligned}
\]
For $c\ge2$, the inequality $\delta_{W_\pi}(c)>\delta_{\mathrm{exp}}(c)$ is equivalent to
\begin{equation}
\label{eq:blowup-rank-inequality}
(r-s-2)c>(r-s)(d-k)-2.
\end{equation}
We further reduce the argument so that we may restrict to $c\ge\max\{2,d\}$.
Indeed, if (i) holds for some $2\le c<d$, then $\delta_{W_\pi}(d)\ge\delta_{W_\pi}(c)>0$, so $d$ also satisfies (i).
For $c\ge\max\{2,d\}$, we have
\[
\delta_{W_\pi}(c)=k-(s+1)(c-d).
\]
Therefore, Condition (i) is equivalent to the existence of $c\ge\max\{2,d\}$ satisfying $k-(s+1)(c-d)>0$ and \eqref{eq:blowup-rank-inequality}. We show the equivalence of the two conditions by dividing into the following three cases.

\begin{itemize}
\setlength{\labelwidth}{\dimexpr\leftmargin-\labelsep\relax}
\renewcommand{\makelabel}[1]{\upshape#1\hfil}

\item[\textbf{Case \(k=0\).}]
We have $\delta_{W_\pi}(c)\le0$, while the right-hand side of the inequality in (ii) is $4-(r-s)\le2\le2d$, so neither condition holds.

\item[\textbf{Case \(s=r-2\).}]
Inequality~\eqref{eq:blowup-rank-inequality} becomes $0>2(d-k)-2\ge0$, which is impossible, so (i) fails.
Condition (ii) also fails because it requires $s<r-2$.

\item[\textbf{Case \(k>0\) and \(s<r-2\).}]
Here $d\ge2$, so we have \(c \geq d\). The inequality $k-(s+1)(c-d)>0$ is equivalent to
\[
c<d+\frac{k}{s+1},
\qquad\text{or equivalently}\qquad
c\le c_*:=d+\left\lceil\frac{k}{s+1}\right\rceil-1.
\]
Since $c_*\ge d$ and $r-s-2>0$, condition (i) holds if and only if \eqref{eq:blowup-rank-inequality} holds at $c=c_*$.
Substituting this value and rearranging, we obtain
\[
\begin{aligned}
&(r-s-2)c_*> (r-s)(d-k)-2
\quad\Longleftrightarrow\quad
2d<(r-s)(k-1)+4+(r-s-2)\left\lceil\frac{k}{s+1}\right\rceil,
\end{aligned}
\]
which proves the equivalence.
\end{itemize}
\end{proof}

Therefore, by  \cite{CelaLian2026}, the condition is equivalent to the unbalancedness of the normal bundle of a general curve in the family under the additional hypotheses $p\ne2$, $d\ge r$, and $d-k\ge r-s-1$. 

\begin{remark}
\label{rmk:converse-rank-deficiency}
The converse to Lemma~\ref{lem:rank-deficiency-implies-unbalancedness} depends on the $H^0$ terms in the normal resolution.
Fix $q\in U$ and $c\ge2$.
Since $H^0(\PP^1,\cE_q\otimes L_c)=0$, the cohomology sequence of \eqref{eq:two-term-resolution-N} gives
\[
0\longrightarrow H^0(\PP^1,\cF_q\otimes L_c)
\longrightarrow H^0(\PP^1,N_{f_q}\otimes L_c)
\longrightarrow\ker(\left.\Theta_{L_c}^1\right|_q)\longrightarrow0,
\]
\[
H^1(\PP^1,N_{f_q}\otimes L_c)
\cong\operatorname{coker}(\left.\Theta_{L_c}^1\right|_q).
\]
If $c\ge\max\{2,\max_{\rho\in\Sigma(1)}d_\rho\}$, then $H^0(\PP^1,\cF_q\otimes L_c)=0$, so
\[
\left.\Theta_{L_c}^1\right|_q\text{ fails to have maximal rank}
\quad\Longleftrightarrow\quad
\begin{cases}
h^0(\PP^1,N_{f_q}\otimes L_c)>0,\\
h^1(\PP^1,N_{f_q}\otimes L_c)>0.
\end{cases}
\]
A bundle on $\PP^1$ is unbalanced if and only if some twist has both cohomology groups nonzero.
Thus unbalancedness implies failure of maximal rank of $\left.\Theta_{L_c}^1\right|_q$ whenever such a twist can be chosen with $c$ in the stated range.
Without the vanishing of $H^0(\PP^1,\cF_q\otimes L_c)$, the nonzero sections of $N_{f_q}\otimes L_c$ may all come from $\cF_q\otimes L_c$ and give no kernel of $\left.\Theta_{L_c}^1\right|_q$.

The numerical criterion in Theorem~\ref{thm:toric-unbalanced-criterion} need not detect every failure of maximal rank, since its proof only gives the lower bounds
\[
\dim\ker(\left.\Theta_{L_c}^1\right|_q)
\ge\dim\ker(\left.A_W\right|_q)\ge\delta_W(c).
\]
For example, in characteristic zero a smooth conic $C\subset\PP^3$ has $N_{C/\PP^3}\cong\cO(4)\oplus\cO(2)$, as follows from the normal sequence for its containing plane and $H^1(\PP^1,\cO(2))=0$.
Here $V$ is one-dimensional and all four Cox divisor degrees equal $2$, so
\[
\begin{aligned}
\delta_V(c)&=c-4\max\{c-2,0\},\\
\delta_{\mathrm{exp}}(c)&=2c-2-4\max\{c-2,0\}.
\end{aligned}
\]
Thus $\delta_V(c)=\delta_{\mathrm{exp}}(c)-c+2\le\delta_{\mathrm{exp}}(c)$ for $c\ge2$, and neither $W=0$ nor $W=V$ satisfies \eqref{eq:toric-unbalanced-criterion}.
Nevertheless, at $c=3$ the bundle $N_{C/\PP^3}(-4)\cong\cO\oplus\cO(-2)$ has $h^0=h^1=1$, so $\left.\Theta_{L_3}^1\right|_q$ has rank $3$ in a $4\times4$ presentation.
Hence a converse to the numerical criterion requires an additional argument establishing generic balancedness whenever the criterion fails.
\end{remark}

\subsection{Classes of jumping loci}
We now apply the determinantal description of the jumping loci given in Section~\ref{subsec:jumping-divisor-class-general} to study how the splitting type of the normal bundle varies within a fixed curve class on \(X_{r,s}\). Corollary~\ref{cor:first-cohomology-jump-class} applied to \(X_{r,s}\) gives Corollary~\ref{cor:blowup-expected-jump-class}, which computes the expected classes of the cohomology-jump loci in terms of \(\zeta_u\) and \(\zeta_v\). As an example,  we compute the first jumping locus for the case \((r,s,d,k)=(3,1,4,2)\); we also exhibit a smooth point in the first jumping locus in every characteristic, which shows that the first jumping locus is nonempty with expected codimension.

Specializing Corollary~\ref{cor:first-cohomology-jump-class} to \(X_{r,s}\), we obtain  the following result.
\begin{corollary}
\label{cor:blowup-expected-jump-class}
For the family of class \((d,k)\) on \(X_{r,s}\), fix \(c\ge2\), and set
\[
a_u=(s+1)\max\{c-d,0\},\qquad
a_v=(r-s)\max\{c-d+k,0\},\qquad
a_w=\max\{c-k,0\}.
\]
Write
\[
C(z)=(1+\zeta_u z)^{a_u}(1+\zeta_v z)^{a_v}
\bigl(1+(\zeta_u-\zeta_v)z\bigr)^{a_w}.
\]
Assume \(r_{L_c}=\min\{3c-2,a_u+a_v+a_w\}>0\), and set
\[
\kappa=\bigl|a_u+a_v+a_w-(3c-2)\bigr|+1.
\]
If \(J_{L_c}^1\) is nonempty and has pure codimension \(\kappa\), its class is
\[
[J_{L_c}^1]=
\begin{cases}
[C(z)]_\kappa,&3c-2\le a_u+a_v+a_w,\\[2pt]
\bigl[C(-z)^{-1}\bigr]_\kappa,&3c-2>a_u+a_v+a_w,
\end{cases}
\qquad\text{in }A^\kappa(U).
\]
\end{corollary}
\begin{proof}
Here \(m=2\), and the pairs \((d_\rho,\zeta_\rho)\) are \((d,\zeta_u)\), \((d-k,\zeta_v)\), and \((k,\zeta_u-\zeta_v)\), with multiplicities \(s+1\), \(r-s\), and \(1\), respectively. Thus the source and target of \(\Theta_{L_c}^1\) have ranks
\(
e=3c-2\) and \(f=a_u+a_v+a_w\). The rest follows from Corollary~\ref{cor:first-cohomology-jump-class}.
\end{proof}

As a concrete example, we consider \((r,s,d,k)=(3,1,4,2)\). Proposition~\ref{prop:small-smooth-witness} establishes in every characteristic that the general normal bundle is balanced and the first jumping locus is a nonempty divisor with a smooth point; its class \(8\zeta_u+4\zeta_v\) is computed using Corollary~\ref{cor:blowup-expected-jump-class}.
\begin{proposition}
\label{prop:small-smooth-witness}
For \((r,s,d,k)=(3,1,4,2)\), in any characteristic, the general normal bundle is \(\cO(6)^{\oplus2}\), and the first jumping locus \(\cJ=J_{L_6}^1\) is a nonempty effective Cartier divisor of class \(8\zeta_u+4\zeta_v\) in \(A^1(U)\). In the Cox order \((u_0,u_1,v_2,v_3,w)\), the point
\[
q_0=(s^4-t^4,\ s^2t^2,\ t^2,\ s^2-st,\ st)
\]
is a smooth point of \(\cJ\) with \(N_{q_0}\cong\cO(5)\oplus\cO(7)\).
\end{proposition}

\begin{proof}
Corollary~\ref{cor:two-term-resolution-blowup} and Lemma--Definition~\ref{thm:universal-pushforward} give \(\operatorname{rk}N_q=2\), \(\deg N_q=12\), and a \(16\times16\) matrix for \(\Theta_{L_6}^1\) satisfying
\[
h^0(N_q\otimes L_6)=\dim\ker(\left.\Theta_{L_6}^1\right|_q)
=16-\operatorname{rank}(\left.\Theta_{L_6}^1\right|_q).
\]
Thus ranks \(16\) and \(15\) correspond respectively to \(\cO(6)^{\oplus2}\) and \(\cO(5)\oplus\cO(7)\).

The coprimeness conditions \(\gcd(v_2,v_3)=\gcd(u_0,u_1,w)=1\) make \(q_0\) an honest map. In the bases of Corollary~\ref{cor:two-term-resolution-blowup}, the minors of \(\Psi_{q_0}\) on rows \((u_1,v_2,w)\), \((u_0,v_2,w)\), and \((u_0,v_3,w)\), evaluated respectively on \(D(st)\), at \(s=0\), and at \(t=0\), are
\[
\begin{cases}
(-4s^2t^4,\,-4t^6,\,-4s^6),&p\ne2,\\
(s^2t^4,\,t^6,\,s^6),&p=2.
\end{cases}
\]
Hence \(q_0\in U\) by Remark~\ref{rmk:unramified-locus-rank-condition}. By openness, the deformation
\begin{equation}
\label{eq:quartic-deformation-family}
q_\varepsilon=(s^4-t^4,\ s^2t^2+\varepsilon s^4,\ t^2,\ s^2-st,\ st)
\end{equation}
lies in \(U\) near \(\varepsilon=0\). Appendices~\ref{app:theta-q0-computation} and~\ref{app:quartic-char-two} show that \(\Theta_{L_6}^1\) has rank \(15\) at \(q_0\) and its determinant has a simple zero along this deformation at \(0\), in every characteristic.

Consequently, the asserted splitting types follow, and Lemma~\ref{lem:jumping-generic-rank} identifies \(\cJ=J_{L_6}^1=V(\det\Theta_{L_6}^1)\). Since \(U\) is smooth and irreducible, this is a nonempty effective Cartier divisor, smooth at \(q_0\). Corollary~\ref{cor:blowup-expected-jump-class}, with \((a_u,a_v,a_w)=(4,8,4)\) and \(\kappa=1\), gives \([\cJ]=4\zeta_u+8\zeta_v+4(\zeta_u-\zeta_v)=8\zeta_u+4\zeta_v\).
\end{proof}

Near \(q_0\), a nonzero \(15\times15\) minor remains nonzero, so the rank is exactly \(15\) along \(\cJ\). Thus \(\cJ\) coincides there with the splitting stratum \(\cJ_{(-1,1)}\), realizing its expected codimension \(u(\cO(5)\oplus\cO(7))=1\) given in Remark~\ref{rmk:deformation-expected-codimension}.

\appendix

\section{The matrix form of \(\left.\Theta_{L_c}^1\right|_q\) in the case \(X = \operatorname{Bl}_{\PP^s}\PP^r\)}
\label{app:theta-q-computation}

Fix a Cox representative \(q=(u_0,\ldots,u_s,v_{s+1},\ldots,v_r,w)\) of class \((d,k)\). Corollary~\ref{cor:two-term-resolution-blowup} gives
\[
\begin{aligned}
\left.\Theta_{L_c}^1\right|_q:\ H^1(\cE_q(-c-1))\longrightarrow{}&
H^1(\cO(d-c-1))^{\oplus(s+1)}\\
&\oplus H^1(\cO(d-k-c-1))^{\oplus(r-s)}
\oplus H^1(\cO(k-c-1)),
\end{aligned}
\]
where all cohomology is on \(\PP^1\), and the source is
\[
H^1(\cE_q(-c-1))\cong
\begin{cases}
H^1(\cO(-c-1))\oplus H^1(\cO(-c))^{\oplus2},&\bar\beta\ne0,\\
H^1(\cO(-c-1))^{\oplus2}\oplus H^1(\cO(1-c)),&\bar\beta=0.
\end{cases}
\]
On each nonzero \(H^1(\cO(-m-1))\), use the basis Serre-dual to \((s^{m-1},s^{m-2}t,\ldots,t^{m-1})\). Order source blocks by the three applicable columns of \(\Psi_q\) and target blocks by the Cox row order. Using \(h^1(\cO(a))=\max\{-a-1,0\}\), for \(c\ge1\) in the nonzero-class case, source block sizes are \(c,c-1,c-1\); for \(c\ge2\) in the zero-class case they are \(c,c,c-2\).
For \(m\geq0\) and a degree-\(e\) form \[f=\sum_{j=0}^{e}f_j s^{e-j}t^j,\] multiplication by \(f\) induces
\[
  H^1(\cO(-m-1))\longrightarrow H^1(\cO(e-m-1)),
\]
whose matrix \(T_m(f)\) is of size \(\max(m-e,0)\times m\), with entries given by
\[
  \bigl(T_m(f)\bigr)_{\alpha\beta}=f_{\beta-\alpha},
  \qquad 0\leq\alpha<m-e,\quad 0\leq\beta<m,
\]
where we set \(f_j=0\) outside \(0\leq j\leq e\).
For \(m>e\), it is the shifted coefficient \((m-e)\times m\)-matrix
\[
  T_m(f)=
  \begin{bmatrix}
    f_0&f_1&\cdots&f_e&0&\cdots&0\\
    0&f_0&f_1&\cdots&f_e&\ddots&\vdots\\
    \vdots&\ddots&\ddots&\ddots&&\ddots&0\\
    0&\cdots&0&f_0&f_1&\cdots&f_e
  \end{bmatrix}.
\]
Each successive row shifts the coefficient vector one place to the right. This is the transpose of multiplication by \(f\) on the Serre-dual monomial bases.

The case conditions and scalar coefficients below are interpreted in the ground field \(k\); the twists and block dimensions use the integer degrees.
In all three cases, each \(u_i\)-block has \(\max(c-d,0)\) rows, each \(v_j\)-block has \(\max(c-d+k,0)\) rows, and the \(w\)-block has \(\max(c-k,0)\) rows; blocks with no rows are omitted.

For \(c\ge1\) and \(e\ge1\), each nonempty derivative block in the first two cases is obtained by shifting one of the coefficient rows
\[
\begin{aligned}
  \partial_sf&:\quad (ef_0,(e-1)f_1,\ldots,f_{e-1}),\\
  \partial_tf&:\quad (f_1,2f_2,\ldots,ef_e)
\end{aligned}
\]
and setting all remaining entries to zero.
For \(e=0\), both derivative blocks are \(c\times(c-1)\) zero matrices.

\begin{itemize}
\setlength{\labelwidth}{\dimexpr\leftmargin-\labelsep\relax}
\renewcommand{\makelabel}[1]{\upshape#1\hfil}

\item[\textbf{Case \(d\ne0\).}]
For \(c\ge1\), the columns \((-A_E,D_s,D_t)\) of Corollary~\ref{cor:two-term-resolution-blowup} give
\begin{equation}
    \label{eq:theta-computation}
  \left[\left.\Theta_{L_c}^1\right|_q\right]=
  \left[
  \begin{array}{c|c|c}
    0&T_{c-1}(\partial_su_0)&T_{c-1}(\partial_tu_0)\\
    \vdots&\vdots&\vdots\\
    0&T_{c-1}(\partial_su_s)&T_{c-1}(\partial_tu_s)\\ \hline
    T_c(v_{s+1})&T_{c-1}(\partial_sv_{s+1})&T_{c-1}(\partial_tv_{s+1})\\
    \vdots&\vdots&\vdots\\
    T_c(v_r)&T_{c-1}(\partial_sv_r)&T_{c-1}(\partial_tv_r)\\ \hline
    -T_c(w)&T_{c-1}(\partial_sw)&T_{c-1}(\partial_tw)
  \end{array}
  \right].
\end{equation}

\item[\textbf{Case \(d=0\) and \(k\ne0\).}]
For \(c\ge1\), use the columns \((A_H,D_s,D_t)\) of Corollary~\ref{cor:two-term-resolution-blowup}. The matrix is obtained from \eqref{eq:theta-computation} by replacing the first column block with blocks \(T_c(u_i)\), \(T_c(v_j)\), and \(0\) in the \(u_i\)-, \(v_j\)-, and \(w\)-rows, respectively.

\item[\textbf{Case \(d=k=0\).}]
For \(c\ge2\), the columns \((A_H,A_E,(H_\rho)_\rho)\)  of Corollary~\ref{cor:two-term-resolution-blowup}  give
\[
\left[\left.\Theta_{L_c}^1\right|_q\right]=
\begin{bmatrix}
T_c(u_i)&0&T_{c-2}(H_{u_i})\\
T_c(v_j)&-T_c(v_j)&T_{c-2}(H_{v_j})\\
0&T_c(w)&T_{c-2}(H_w)
\end{bmatrix}_{i,j}.
\]
For a coordinate \(f\) of degree \(e \geq 2\) such that \(e = 0\) in the field \(k\),
the factored derivative is
\[
H_f=-\sum_{j=0}^{e-2}(j+1)f_{j+1}s^{e-2-j}t^j.
\]
It satisfies \(\partial_sf=tH_f\) and \(\partial_tf=-sH_f\). For a constant coordinate or one with no sections, \(H_f=0\), and its block is the zero matrix with the prescribed target and source dimensions. Multiplication by \(H_f\) uses the same constant-diagonal matrix construction, with no division by \(e\).
\end{itemize}

\subsection{The quartic family: \(p\ne2\)}
\label{app:theta-q0-computation}

For the family \(q_\varepsilon\) in \eqref{eq:quartic-deformation-family}, let \(M_\varepsilon\) be the \(16\times16\) matrix obtained from \eqref{eq:theta-computation} with \(c=6\). Write its source coordinates as \((a_0,\ldots,a_5;b_0,\ldots,b_4;c_0,\ldots,c_4)\); the target blocks follow the order \((u_0,u_1,v_2,v_3,w)\). All calculations below hold over \(\mathbb Z[1/2][\varepsilon]\).

Let \(B\) consist of the fourteen rows belonging to \(u_0,v_2,v_3,w\). It is independent of \(\varepsilon\), and its minor obtained by deleting columns \(b_1,c_1\) has determinant \(256\). Thus \(\ker B\) is free of rank two with coordinates \(b_1,c_1\). Solving \(Bx=0\) and substituting into the two remaining \(u_1\)-rows gives
\[
\left.M_\varepsilon\right|_{\ker B,\,u_1}
=4\begin{pmatrix}0&4\varepsilon-1\\ \varepsilon&-2\end{pmatrix}.
\]
Moving the \(u_1\)-rows and the columns \(b_1,c_1\) to the end uses even permutations. The Schur complement therefore gives, in the original fixed bases,
\[
\det M_\varepsilon
=256\cdot16\varepsilon(1-4\varepsilon)
=4096\varepsilon(1-4\varepsilon).
\]
The remaining block has rank one at \(\varepsilon=0\), so \(\operatorname{rank}M_0=15\).

\subsection{The quartic family: \(p=2\)}
\label{app:quartic-char-two}

For the same family, form \(M_\varepsilon\) using the \(d=k=0\) block formula above with \(c=6\). The source coordinates are now \((a_0,\ldots,a_5;b_0,\ldots,b_5;c_0,\ldots,c_3)\), corresponding to \((A_H,A_E,(H_\rho)_\rho)\), and the target order is unchanged.

Delete the second row of the \(u_1\)-block, namely the fourth row of \(M_\varepsilon\), to obtain \(B'\). Its minor obtained by also deleting column \(a_1\) has determinant \(1\) over \(k[\varepsilon]\). Hence \(\ker B'\) is free of rank one with coordinate \(a_1\). Solving \(B'x=0\) gives \(a_3=0\), so the omitted row restricts to \(\varepsilon a_1+a_3=\varepsilon a_1\). Consequently, in the original fixed bases,
\[
\det M_\varepsilon=\varepsilon,
\qquad \operatorname{rank}M_0=15.
\]
In both characteristic cases, the determinant has a simple zero at \(\varepsilon=0\). These are the determinant and rank identities used in Proposition~\ref{prop:small-smooth-witness}.

\end{document}